\documentclass[12pt]{article}
\usepackage[utf8]{inputenc}

\PassOptionsToPackage{hyphens}{url}
\usepackage[colorlinks, citecolor=blue, linkcolor=blue]{hyperref}
\usepackage{amsmath, amsfonts, amssymb, amsthm, graphicx, enumerate}
\usepackage{caption, subcaption}

\usepackage[numbers]{natbib}
\usepackage{authblk, orcidlink}
\usepackage[margin=1in]{geometry}
\usepackage{setspace}
\theoremstyle{plain}
\newtheorem{proposition}{Proposition}
\newtheorem{theorem}{Theorem}

\newtheorem{corollary}{Corollary}

\theoremstyle{remark}
\newtheorem{assumption}{Assumption}

\newtheorem{example}{Example}

\newcommand{\norm}[1]{\left\lVert#1\right\rVert}
\newcommand{\B}{\mathcal{B}}
\newcommand{\W}{\mathcal{W}}
\newcommand{\e}{\epsilon}
\newcommand{\w}{\omega}

\newcommand{\Z}{\mathbb{Z}}
\newcommand{\E}{\mathbb{E}}

\newcommand{\C}{\mathcal{C}}
\newcommand{\R}{\mathbb{R}}

\let\temp\phi
\let\phi\varphi
\let\varphi\temp  

\begin{document}

\title{Lower bounds on the rate of convergence for accept-reject-based Markov chains in Wasserstein and total variation distances}

\author[1]{Austin Brown \orcidlink{0000-0003-1576-8381} \thanks{ad.brown@utoronto.ca}}
\author[2]{Galin L. Jones \orcidlink{0000-0002-6869-6855} \thanks{galin@umn.edu}}

\affil[1]{Department of Statistical Sciences, University of Toronto, Toronto, Ontario, Canada}
\affil[2]{School of Statistics, University of Minnesota, Minneapolis, MN, U.S.A.}

\maketitle

\begin{abstract} 
To avoid poor empirical performance in  Metropolis-Hastings and other accept-reject-based algorithms practitioners often tune them by trial and error. Lower bounds on the convergence rate are developed in both total variation and Wasserstein distances in order to identify how the simulations will fail so these settings can be avoided, providing guidance on tuning. Particular attention is paid to using the lower bounds to study the convergence complexity of accept-reject-based Markov chains and to constrain the rate of convergence for geometrically ergodic Markov chains. The theory is applied in several settings. For example, if the target density concentrates with a parameter $n$ (e.g. posterior concentration, Laplace approximations), it is demonstrated that the convergence rate of a Metropolis-Hastings chain can be arbitrarily slow if the tuning parameters do not depend carefully on $n$. This is demonstrated with Bayesian logistic regression with Zellner's g-prior when the dimension and sample increase together and flat prior Bayesian logistic regression as $n$ tends to infinity.
\end{abstract}



\section{Introduction}
\label{sec:intro}

Metropolis-Hastings algorithms \citep{hast:1970, metr:1953, Tierney1998} are foundational to the use of Markov chain Monte Carlo (MCMC) methods in statistical applications \citep{broo:etal:2011}. Metropolis-Hastings requires the choice of a proposal distribution and this choice is crucial to the efficacy of the resulting sampling algorithm.  The choice of tuning parameters for the proposal distribution prompted ground-breaking research into optimal scaling  \citep{robe:weak:1997} and adaptive MCMC methods \citep{Haario2001}. 
 Nevertheless, in many applications, this remains a difficult task and is often accomplished by trial and error.  

Metropolis-Hastings Markov chains are an instance of accept-reject-based (ARB) Markov chains, which require a proposal distribution and either accept or reject a proposal at each step. While Metropolis-Hastings chains are the most widely used and studied ARB chains, others are of practical importance or theoretical interest.  This includes, among others, non-reversible Metropolis-Hastings \citep{bier:2015}, Barker's algorithm \citep{bark:1964}, portkey Barker's algorithm \citep{vats:portkey:2022}, and the Lazy-Metropolis-Hastings algorithm \citep{lats:robe:2013}. The main difference between varieties of ARB chains is the definition of the acceptance probability, but the choice of a proposal remains crucial for the efficacy of the associated sampling algorithm. 
Indeed, there has been recent interest in the optimal scaling problem for more general ARB chains \citep{agra:etal:2021}.

Under standard regularity conditions, ARB chains converge to their target distribution and thus will eventually produce a representative sample from it. If the proposal is not well chosen, then this convergence can take prohibitively long. A significant goal is to provide a tool for identifying when an ARB chain will \textit{fail} to produce a representative sample within any reasonable amount of time. This is achieved by developing novel lower bounds on the convergence rate.   

The lower bounds lead to similar conclusions whether the convergence is measured in total variation or a Wasserstein distance, but the technical conditions required for the Wasserstein distances are slightly stronger than those for total variation.
Most commonly, the convergence of ARB chains has been measured with total variation \citep{Rosenthal1995, Tierney1994}. However more recently, the Wasserstein distances \citep{Kantorovich1957, Villani2003, Villani2008} from optimal transportation have been used in the convergence analysis of high-dimensional MCMC algorithms \citep{Dalalyan2017, Durmus2015, Hairer2014, qin2019convergence, raja:spar:2015}.
The interest in Wasserstein distances is due to the observation that Wasserstein distances may enjoy improved scaling properties in high dimensions compared to the total variation distance. Properties of Wasserstein distances which are beneficial to scaling to high dimensions also make developing lower bounds more challenging. Compared to the relatively straightforward proofs for total variation lower bounds, new techniques are required to develop lower bounds in Wasserstein distances and hence this is the main technical contribution. The practical importance is that when the lower bounds show an ARB Markov chain has poor convergence properties in total variation, then they also show it will have poor convergence properties in Wasserstein distances.
 
Convergence rate analysis of MCMC chains on general state spaces largely has centered on establishing upper bounds on the convergence rate \citep{Jones2001, Rosenthal1995} or, failing that, in establishing qualitative properties such as geometric ergodicity, which ensures an exponential convergence rate to the target distribution in either total variation or Wasserstein distances  \citep{Hairer2011, JonesHobert2004, Meyn2009, Rosenthal1995, Tierney1994}.  There has been important work establishing geometric ergodicity for Metropolis-Hastings chains \cite[e.g.][]{Jarner2000, Mengersen1996, Roberts1996geo}, but, even though there is some recent work for specific random-walk Metropolis-Hastings chains \citep{andr:etal:MH:2022, bhat:jone:2023}, it has not yet produced much understanding of their quantitative rates of convergence.  The lower bounds developed here allow the exponential rate of convergence to be bounded below and hence, in conjunction with the existing approaches, will help provide a clearer picture of the quantitative convergence properties of geometrically ergodic ARB chains. 

Convergence complexity of MCMC algorithms with respect to the dimension $d$ and sample size $n$ has been of significant recent interest \citep{Belloni2009, Austin2022, KarlOskar2019, Johndrow2019, qin2019convergence, raja:spar:2015, Yang2017, yang:etal:2016, JZhuo2021}.  With one exception \cite{Austin2022}, this work has not considered the setting where both $n$ and $d$ are large for Metropolis-Hastings algorithms. Lower bounds can be used to study conditions on the tuning parameters which imply the convergence rate can be arbitrarily slow as the dimension $d$ and sample size $n$ increase, even if the chain is geometrically ergodic for a fixed $d$ and $n$.  This can then be used to choose the scaling so that the poor high-dimensional convergence properties might be avoided. One implication of the theory developed below is that the scaling should be chosen explicitly as a function of both $n$ and $d$, with the details dependent upon the setting, since if they are chosen otherwise, the resulting ARB chain can have poor convergence properties. In comparison, optimal scaling results are derived in the setting where the dimension $d \to \infty$, but the explicit scaling dependence on $n$ is often infeasible to compute.

Throughout, the theory is illustrated on examples of varying complexity, some of which concern ARB chains that are not versions of Metropolis-Hastings. Several significant, practically relevant applications of Metropolis-Hastings are studied.
The lower bounds are studied on the random walk Metropolis-Hastings (RWMH) algorithm for general log-concave target densities.  In particular, an application to RWMH for Bayesian logistic with Zellner's g-prior is studied as $d, n \to \infty$. In this example, the convergence rate can be arbitrarily slow if the scaling does not depend carefully on $d$ and $n$. 
Following that, the focus turns to a class of Metropolis-Hastings algorithms using a general Gaussian proposal which in special cases is the proposal used in RWMH \citep{Tierney1994}, Metropolis-adjusted Langevin algorithm (MALA) \citep{robe:rose:1998:optimal, Roberts1996Langevin}, and other variants such as Riemannian manifold MALA \citep{Girolami2011}.
Metropolis-Hastings algorithms are considered for more general target densities under conditions which imply the densities concentrate towards their maximum point with $n$ (e.g. Bayesian posterior concentration, Laplace approximations) as $d, n \to \infty$. Once again, in this general setup, if the tuning parameters do not depend carefully on $n$, the convergence rate can be arbitrarily slow as $d, n \to \infty$. As an application of the general result, flat prior Bayesian logistic regression is studied as $n \to \infty$.

General lower bounds on the convergence rate have not been previously investigated for ARB chains outside of some Metropolis-Hastings independence samplers \citep{Austin2022, Roberts2011, Wang2020}. 
Others have instead focused on upper bounding the spectral gap using the conductance \citep{AndrieuVihola2015, Rabee2020, Hairer2014, JarnerYuen2004, Johndrow2019, Lawler1988, Schmidler2011,  Woodard2009b, Woodard2014}. An upper bound on the spectral gap lower bounds the convergence rate. This requires a reversible Markov chain \citep{robe:rose:1997, Lawler1988}, which is not assumed in the development of the general lower bounds below.
For reversible ARB Markov chains, there are some new comparisons to the lower bounds in total variation and Wasserstein distances to the conductance.
In order to make such comparisons, an equivalence between the spectral gap and a weaker Wasserstein convergence rate for general reversible Markov chains is developed, which is of independent interest.

Recently, a significant related research interest has emerged regarding the dimension dependence on the mixing time for the MALA algorithm \citep{Chewi2021, Dwivedi2018, Wu2021, Lee2021}. 
Lower bounds on the total variation mixing times were shown for MALA in some specific examples \citep{Chewi2021, Wu2021, Lee2021}.
A more general lower bound on the mixing time in the chi-square divergence was also developed for reversible Markov chains \citep{Wu2021}.
In comparison, the lower bounds hold in Wasserstein distances and total variation which are more commonly used in modern convergence analysis and are developed for more general Metropolis-Hastings algorithms than MALA.
Moreover, lower bounds on the convergence rates provide important information compared to mixing times as convergence rates are used to upper bound the asymptotic variance in the Markov chain central limit theorem \citep{Jones2004} and mean squared error \citep{Krzysztof2013, Rudolf2012} as well as in concentration inequalities \citep{Joulin2010}.

The remainder is organized as follows. Section~\ref{sec:arb_mc} provides a formal definition of ARB Markov chains and some relevant background.  Total variation lower bounds are studied in Section~\ref{sec:tv_lb}, including a comparison with conductance methods. Then Section~\ref{sec:wass_lb} presents lower bounds in Wasserstein distances along with a comparison to conductance methods. Applications of the theoretical results are considered in Section~\ref{sec:applications}. Section~\ref{sec:componentwise} considers random scan componentwise ARB algorithms and develops extensions of the general lower bounds in both total variation and Wasserstein distances. Final remarks are given in Section~\ref{sec:final}.  Most proofs are deferred to the Appendix.

\section{Accept-reject-based Markov chains}
\label{sec:arb_mc}

Let $\Omega$ be a nonempty Borel measurable space and $\B(\Omega)$ the Borel sigma algebra on $\Omega$.
Suppose $\Pi$ is a probability measure on $(\Omega, \B(\Omega))$ supported on a nonempty measurable set $\Theta \subseteq \Omega$ having density $\pi$ with respect to a sigma-finite measure $\lambda$ on $(\Omega, \B(\Omega))$.
The support of $\pi$ is $\Theta$ so for each $\theta \in \Theta$, $\pi(\theta) > 0$.
Assume for each $\theta \in \Theta$, $\{ \theta \}$ is Borel measurable with $\lambda(\{ \theta \}) = 0$ and $\lambda$ is strictly positive (e.g., Lebesgue measure) meaning for nonempty, measurable open sets $U \subseteq \Theta$, $\lambda(U) > 0$.  While this general setting will suffice for most of the work, at various points more specific assumptions on $\Omega$ will be required.

For each $\theta \in \Omega$, suppose that $Q(\theta, \cdot)$ is a proposal distribution on $(\Omega, \B(\Omega))$ having density $q(\theta, \cdot)$ with respect to $\lambda$. If $a : \Omega \times \Omega \to [0,1]$, define the acceptance probability by
$$
A(\theta) = \int a(\theta, \theta') q(\theta, \theta') d\lambda(\theta').
$$
Let $\delta_\theta$ denote the Dirac measure at the point
$\theta$.
The accept-reject-based Markov kernel $P$ is defined on $(\Omega, \B(\Omega))$ for $\theta \in \Omega$ and measurable sets $B \subseteq \Omega$ as
$$
P(\theta, B) = \int_{B} a(\theta, \theta') q(\theta, \theta') d\lambda(\theta') + \delta_{\theta}(B) [1-A(\theta)].
$$
Let $\Z_+$ denote the set of positive integers.
The Markov  kernel for iteration time $t \in \Z_+$ with $t \ge 2$ is
defined recursively by 
\[
P^t(\theta, \cdot)
= \int P^{t-1}(\cdot, \cdot) dP(\theta, \cdot)
\]
where $P \equiv P^1$.

It is assumed that $a$ is such that $\Pi$ is invariant for $P$.
Invariance is often ensured through a detailed balance (i.e., reversibility) condition, but this is not required in general. 
For example, detailed balance does not hold for non-reversible Metropolis-Hastings yet $\Pi$ is invariant \citep{bier:2015}. 
Of course, detailed balance holds for Metropolis-Hastings (MH) where
\begin{equation}
\label{eq:a_mh}
a_{MH}(\theta, \theta') = 
\begin{cases}
\frac{ \pi(\theta') q(\theta', \theta) }{\pi(\theta)
    q(\theta, \theta')} \wedge 1, \text{ if }\pi(\theta) q(\theta, \theta') > 0 \\ 
1, \hspace{2cm} \text{ if } \pi(\theta) q(\theta, \theta') = 0.
\end{cases}
\end{equation}  

Define the total variation distance between probability measures $\mu, \nu$ on $( \Omega, \B(\Omega) )$ by 
$
\norm{\mu - \nu}_{\text{TV}}
= \sup_{f \in \mathcal{M}_1}
\left[ \int f d\mu - \int f d\nu \right]
$
where $\mathcal{M}_1$ is the set of Borel measurable functions $f : \Omega \to [0, 1]$.
It is well-known that if $P$ is $\Pi$-irreducible, aperiodic, and Harris recurrent, the ARB Markov chain converges in total variation, as $t \to \infty$,  to the invariant distribution \citep{Tierney1994, Tierney1998} and hence will eventually produce a representative sample from $\Pi$. 
Section~\ref{sec:tv_lb} considers measuring this convergence with the total variation norm, while Section~\ref{sec:wass_lb} considers doing so with Wasserstein distances from optimal transportation.

\section{Total variation lower bounds}
\label{sec:tv_lb}

\subsection{Lower bounds}
\label{sec:lb_comp}

Previously, a lower bound in the total variation distance was shown for the Metropolis-Hastings independence sampler \citep[Lemma 1]{Wang2020}. A similar argument provides a lower bound for general ARB chains.

\begin{theorem}
\label{thm:lb}
For every $t \in \Z_+$ and every $\theta \in \Theta$,
\[
\norm{P^t(\theta, \cdot) - \Pi}_{\text{TV}}
\ge \left[ 1 - A(\theta) \right]^t.
\]
\end{theorem}

Note that this bound is sharp for the specific case of the Metropolis-Hastings independence sampler \cite{Wang2020}. Although it has been suppressed here, $A(\theta)$ often depends on the sample size $n$ and the dimension $d$; for example, consider the setting where $\pi$ is a Bayesian posterior density.  If $A(\theta) \to 0$ as $n \to \infty$ or $d \to \infty$, then the total variation lower bound will approach 1 and the ARB Markov chain will be ineffective in that regime.  However, once this has been identified, the proposal may be chosen so as to avoid the difficulties.  

Of course, $A(\theta)$ can be difficult to calculate analytically.  However, this is often unnecessary since it typically suffices to study an upper bound for it.  If a finer understanding is required, it is straightforward to use Monte Carlo sampling to produce a functional estimate of $A(\theta)$.  Both of these approaches will be illustrated 
 in several examples later. Consider the following simple example of a non-reversible Metropolis-Hastings Markov chain \citep[Section 4.3]{bier:2015}.

\begin{example} 
For $d \in \Z_+$, let $\R^d$ be the Euclidean space of dimension $d$.
Denote the standard $p$-norms by $\norm{\cdot}_p$.
Denote the Gaussian distribution on $\R^d$ with mean $\mu$ and symmetric, positive-definite covariance matrix $C$ by $N(m, C)$.
For $h \in (0, 1]$, consider the Crank-Nicolson proposal $N(\sqrt{1 - h} \theta, h I_d)$ so that the proposal has the standard Gaussian density as its invariant density, which is denoted $\rho$. Suppose $\pi$ is a Gaussian density of the distribution $N(0, \sigma^2 I_d)$ with $\sigma^2 > 1$.
For $c = 1/(2^d\sigma^d)$,
\[
c \rho(\theta) / \pi(\theta)
= \frac{1}{2^d} \exp\left(-\frac{1}{2} (1 - 1/\sigma^2) \theta^T \theta\right)
\in (0, 1].
\]
If
\[
\gamma(\theta, \theta')
= c \left[ \frac{\rho(\theta)}{\pi(\theta)} \pi(\theta) q(\theta, \theta') - \frac{\rho(\theta')}{\pi(\theta')} \pi(\theta') q(\theta', \theta) \right],
\]
then $\gamma$ defines a valid non-reversible Metropolis-Hastings Markov chain \cite[Section 4.1]{bier:2015} since
\begin{align*}
\gamma(\theta, \theta') 
+ \pi(\theta') q(\theta', \theta)
&= c \frac{\rho(\theta)}{\pi(\theta)} \pi(\theta) q(\theta, \theta') 
+ \left[ 1 - c \frac{\rho(\theta')}{\pi(\theta')} \right] 
\pi(\theta') q(\theta', \theta)
\ge 0.
\end{align*}  
If $A_{NR}$ and $A_{MH}$ are the acceptance probabilities for non-reversible MH and the usual MH with the same proposal, respectively, then 
\begin{align*}
A_{NR}(\theta)
&= \int \left[ \frac{\gamma(\theta, \theta') + \pi(\theta') q(\theta', \theta)}{\pi(\theta) q(\theta, \theta')}  \wedge 1 \right]
q(\theta, \theta') d\theta'
\\
&\le 2^{-d} + A_{MH}(\theta).
\end{align*}
For large $d$, the lower bounds of Theorem~\ref{thm:lb} can be similar for both Markov chains. However, for small values of $d$, the lower bound for non-reversible MH can be appreciably smaller than the lower bound for MH. That is, in low dimensions, even if 
$A_{MH}(\theta) \approx 0$ so that the MH algorithm converges slowly, the non-reversible MH may enjoy a much faster convergence rate. 
\end{example}

In general, it is difficult to compare ARB Markov chains based on lower bounds.  However, there are some settings where it may be informative.  

\begin{corollary}
\label{cor:compare}
Let $P_1$ and $P_2$ be ARB kernels with acceptance functions $a_1$ and $a_2$, respectively.   If $a_{1}(\theta, \theta') \ge a_{2}(\theta, \theta')$ for all $\theta, \theta'$, then
$$
\norm{P_{2}^{t} (\theta, \cdot) - \Pi }_{\text{TV}} \ge [1 - A_{2}(\theta)]^t \ge [1 - A_{1}(\theta)]^t.
$$
\end{corollary}

If $A_{1}(\theta) \to 0$ (perhaps as $d,n \to \infty$), then the
lower bound for both $P_1$ and $P_2$ will tend to 1.  Of course, if
$A_{1}(\theta) \to c >0$, then it may still be the case that
$A_{2}(\theta) \to 0$. That is, $P_1$ might avoid poor convergence
properties while $P_2$ does not. 

\begin{example}
Consider the portkey Barker's kernel \citep{vats:portkey:2022} 
where if $d(\theta, \theta') \ge 0$ is symmetric so that $d(\theta, \theta') = d(\theta', \theta)$, then
\begin{equation}
\label{eq:a_pb}
a_{PB}(\theta, \theta') = \frac{ \pi(\theta') q(\theta', \theta) }{\pi(\theta)
    q(\theta, \theta') + \pi(\theta') q(\theta', \theta) + d(\theta, \theta')}.
\end{equation}
Notice that Barker's \citep{bark:1964} algorithm is recovered if $d(\theta, \theta') =0 $. 

It is well-known that Metropolis-Hastings is more efficient than Barker's in the Peskun sense \citep{pesk:1973} so that the variance of the asymptotic normal distribution for a sample mean will be larger if the Monte Carlo sample is produced using Barker's algorithm.  However, the asymptotic variance is only greater by a factor of 2 \citep{lats:robe:2013} and there are some settings where portkey Barker's is preferred \citep{vats:portkey:2022}.

Notice that if the same proposal density is used for Metropolis-Hastings and portkey Barker's, then
\[
  a_{PB}(\theta, \theta') \le a_{MH}(\theta, \theta') \le 1
\]
and hence, by Corollary~\ref{cor:compare},
\[
\norm{P_{PB}^t(\theta, \cdot) - \Pi}_{\text{TV}} \ge (1-A_{PB}(\theta))^t \ge (1-A_{MH}(\theta))^t . 
\]
\end{example}

\subsection{Geometric ergodicity}

 The kernel $P$ is $(\rho, M)$-geometrically ergodic if there is a   $\rho \in (0, 1)$ and a function $\theta \mapsto M(\theta) \in (0, \infty)$ such that for every $\theta \in \Theta$ and every $t \in \Z_+$,
$$
\norm{P^t(\theta, \cdot) - \Pi(\cdot)}_{\text{TV}}
\le M(\theta) \rho^t.
$$ 

There has been substantial effort put into establishing the existence of $\rho$ for some ARB Markov chains \cite[see, e.g.,][]{Jarner2000, Mengersen1996, Roberts1996Langevin, Roberts1996geo}, but, outside of the Metropolis-Hastings independence sampler \citep{Wang2020} and some random walk MH algorithms \cite{andr:etal:MH:2022, bhat:jone:2023} these efforts have not yielded constraints on $\rho$.

\begin{theorem}
\label{thm:lb_rate}
If $P$ is $(\rho, M)$-geometrically ergodic, then 
\[
\rho \ge 1 - \inf_{\theta \in \Theta} A(\theta).
\]
\end{theorem}

\begin{example}
Let $b > 1$ and consider the following Gaussian densities
$$
f(x,y) = \frac{b}{\pi} e^{-(x^2 + b^2 y^2)} \quad \text{ and } \quad g(x,y) = \frac{b}{\pi} e^{-(b^2x^2 + y^2)} .
$$
Set 
$$
\pi(x,y) = \frac{1}{2} f(x,y) + \frac{1}{2} g(x,y)
$$
and consider a random walk MH kernel with a Gaussian proposal centered at the previous step and scale matrix $ h I_2$. This RWMH is geometrically ergodic \citep{Jarner2000}, but only the existence of $\rho < 1$ has been established.   

For any $(x,y)$
$$
A(x,y)  \le  \iint\frac{\pi(u,v)}{\pi(x,y)}  q( (x, y), (u,v) ) du dv 
\le  \frac{1}{b h [ e^{-(x^2 + b^2 y^2)} + e^{-(b^2x^2 + y^2)}]} .
$$
By Theorem~\ref{thm:lb_rate},
$$
\rho \ge 1 - \frac{1}{2bh} .
$$
\end{example}

\subsection{Comparison with conductance methods}
\label{sec:compare_conductance}

In this section, the ARB Markov kernel is assumed to be reversible with respect to $\Pi$. 
Let $L^p(\mu)$ denote the Lebesgue spaces with respect to a measure $\mu$.
Conductance can be used to lower bound the convergence rate and this is compared to the techniques used above.
If a Markov kernel $P$ satisfies
\[
\sup_{
\substack{f \in L^2(\Pi) \\ 
\norm{f - \int f d\Pi}_{L^2(\Pi)} \not= 0}
} \frac{\norm{P f - \int f d\Pi}_{L^2(\Pi)}}{\norm{f - \int f d\Pi}_{L^2(\Pi)}} = \beta \in (0, 1),
\]
then it is said to have a $1 - \beta$ spectral gap.
For measurable sets $B$ with $0 < \Pi(B) < 1$, define
\[
k_P(B) = \frac{\int_B P(\cdot, B^c) d\Pi}{\Pi(B) \Pi(B^c)}
\]
and define the conductance 
$k_P = \inf\{ k_P(B) :0 < \Pi(B) < 1 \}$ \citep{Lawler1988}.
The conductance can be used to upper bound the spectral gap \citep[Theorem 2.1]{Lawler1988}:
\[
1 - \beta \le k_P.
\]

Since $P$ is a reversible Markov kernel, there is an equivalence between a spectral gap and geometric convergence in total variation \citep[Theorem 2.1]{robe:rose:1997}. Specifically, there is a $\rho \in (0, 1)$ such that for every probability measure $\mu$ with $d\mu / d\Pi \in L^2(\Pi)$, there is a constant $C_\mu \in (0, \infty)$ such that
\[
\norm{\mu P^t - \Pi}_{\text{TV}} \le C_\mu \rho^t
\]
if and only if there is a spectral gap with $\beta \le \rho$.
Under these conditions, the convergence rate can be lower bounded by the conductance so that
\[
1 - \rho \le k_P.
\]
Under further conditions on the function $M$, the conductance will also lower bound the convergence rate if $P$ is $(\rho, M)-$geometrically ergodic \citep[Proposition 2.1 (ii), Theorem 2.1]{robe:rose:1997}.

For sets $B$ with $0 < \Pi(B) \le 1/2$, to move from $B$ to $B^c$, the Metropolis-Hastings kernel $P$ must accept and \citep[Proposition 2.16]{Hairer2014} 
\[
k_P
\le 2 \sup_{x \in B} A(x).
\]

\begin{proposition}
\label{prop:conductance_accept_ub}
Assume $(\Omega, d)$ is a nonempty metric space.
Let $P$ be a reversible ARB Markov kernel with an upper semicontinuous acceptance probability $A(\cdot)$.
Then 
\[
k_P \le \inf_{\theta \in \Theta} A(\theta).
\]
\end{proposition}


The results of this section imply that under the stronger conditions of reversibility, that $\int M^2 d\Pi < \infty$, and  $A(\cdot)$ is upper semicontinuous, existing results \cite[Proposition 2.1 (ii), Theorem 2.1]{robe:rose:1997} combined with Proposition~\ref{prop:conductance_accept_ub} will yield the conclusion of Theorem~\ref{thm:lb_rate}.

\section{Wasserstein lower bounds}
\label{sec:wass_lb}

Wasserstein distances have become popular in the convergence analysis of high-dimensional MCMC algorithms due to improved scaling properties in high dimensions compared to the total variation distance. Where total variation controls the bias of bounded, measurable functions, Wasserstein distances control the bias of Lipschitz functions. The regularity of Lipschitz functions compared to bounded functions contributes to improved scaling in high dimensions but makes developing lower bounds for Wasserstein distances more challenging. In particular, the relatively straightforward technique of Theorem~\ref{thm:lb} cannot be used and new techniques are required. However, lower bounds are available for many Wasserstein distances which are analogous to the lower bounds from Section~\ref{sec:tv_lb} and which can result in similar conclusions if the acceptance probability is not well-behaved.

The Wasserstein lower bounds also yield necessary conditions for Wasserstein geometric ergodicity of ARB chains that are similar to those available for total variation \citep{Roberts1996geo}.  A result of independent interest establishes the equivalence between a spectral gap and Wasserstein geometric ergodicity for general reversible Markov chains, complementing existing results for total variation \citep{robe:rose:1997}.

\subsection{Lower bounds}
\label{sec:w_lb_compare}

For probability measures $\mu, \nu$, let $\C(\mu, \nu)$ be the set of all joint probability measures with marginals $\mu, \nu$.
With a metric cost function $c(\cdot, \cdot)$, the $c$-Wasserstein distance of order $p \in \Z_+$ is 
\[
  \W_{c}^p\left( \mu, \nu \right) =  \left( \inf_{\Gamma \in
    \C(\mu, \nu)} \int
  c(\theta, \w)^p d\Gamma(\theta, \w) \right)^{1/p}.
\]

\noindent Consider the following assumption on the target density.
\begin{assumption}
\label{assumption:ub_pi}
Suppose $\Omega = \R^d$ where $d \in \Z_+$ and $\Theta \subseteq \mathbb{R}^d$.
Suppose $\pi$ is a density with respect to Lebesgue measure and suppose $s = \sup_{\theta \in \Theta} \pi(\theta) < \infty$.
\end{assumption}

\noindent The Wasserstein distances are required to satisfy the following condition.
\begin{assumption}
\label{assumption:wass_condition}
Suppose for some constant $C_{0, d}$, $c(\cdot, \cdot) \ge C_{0, d} \norm{\cdot - \cdot}_1$.
\end{assumption}
\noindent As norms are equivalent on $\R^d$, these cost metrics include any norm $\norm{\cdot}$ on $\R^d$.

For the special case of the Metropolis-Hastings independence sampler, there are existing lower bounds for some Wasserstein distances \citep{Austin2022}.  The above assumptions are enough to ensure lower bounds for the general ARB Markov chain and a wider class of Wasserstein distances.

\begin{theorem}
\label{thm:wasserstein_lb}
Let assumptions \ref{assumption:ub_pi} and \ref{assumption:wass_condition} hold. If $C_{d, \pi} = C_{0, d} d [ 2 s^{\frac{1}{d}}
  ( 1 + d )^{1 + \frac{1}{d}} ]^{-1}$, then for every $t \in \Z_+$ and every $\theta \in \R^d$
\[
\W_{c}^p(P^t(\theta, \cdot), \Pi)
\ge C_{d, \pi} [ 1 - A(\theta) ]^{t \left( 1 + \frac{1}{d} \right)}.
\]
\end{theorem}

\begin{corollary}
Let $P_1$ and $P_2$ be ARB kernels with acceptance functions $a_1$ and $a_2$, respectively.   If assumptions \ref{assumption:ub_pi} and \ref{assumption:wass_condition} hold and $a_{1}(\theta, \theta') \ge a_{2}(\theta, \theta')$ for all $\theta, \theta'$, then
\[
\W_{c}^p(P_2^t(\theta, \cdot), \Pi) \ge C_{d, \pi} [ 1 - A_2(\theta) ]^{t \left( 1 + \frac{1}{d} \right)} \ge C_{d, \pi} [ 1 - A_1(\theta) ]^{t \left( 1 + \frac{1}{d} \right)}.
\]
\end{corollary}

\subsection{Geometric ergodicity}

Say $P$ is $(c, p, \rho, M)$-geometrically ergodic if there is a $\rho \in (0, 1)$ and a function $\theta \mapsto M(\theta) \in (0, \infty)$ such that for every $\theta \in \Theta$ and every $t \in \Z_+$,
\[
\W_{c}^p\left( P^t(\theta, \cdot), \Pi \right)
\le M(\theta) \rho^t.
\]
Compared to the definition of geometric ergodicity in total variation, this requires a metric cost function $c$ for the Wasserstein distance of order $p$. 

\begin{theorem}
\label{thm:wass_lb_rate}
Let assumptions~\ref{assumption:ub_pi} and~\ref{assumption:wass_condition} hold.
If $P$ is $(c, p, \rho, M)$-geometrically ergodic, then
\[
\rho^{\frac{d}{d + 1}} \ge 1 - 
\inf_{\theta \in \Theta
}A(\theta).
\]
\end{theorem}

Recall that if the rejection probability cannot be bounded below one, that is, if the acceptance probability satisfies $\inf_{\theta \in \R^d} A(\theta) = 0$, then a Metropolis-Hastings algorithm fails to be geometrically ergodic \citep[Proposition 5.1]{ Roberts1996geo}. 
Theorem~\ref{thm:wass_lb_rate}  yields the same in these Wasserstein distances for ARB kernels.
There are well-known examples where the condition on this happens 
\cite[Example 1]{Mengersen1996}, \cite[][Proposition 5.2]{Roberts1996geo}.

\begin{corollary}
Let assumptions~\ref{assumption:ub_pi} and~\ref{assumption:wass_condition} hold.
If $\inf_{\theta \in \Theta} A(\theta) = 0$, then $P$ cannot be $(c, p, \rho, M)$-geometrically ergodic.
\end{corollary}
\begin{proof}
If it were $(c, p, \rho, M)$-geometrically ergodic, then, by Theorem~\ref{thm:wass_lb_rate}, $\rho^{\frac{d}{d + 1}} \ge 1$, which is a contradiction.
\end{proof}

Analogous comparisons can be made to conductance methods as in Section~\ref{sec:compare_conductance}. 
In particular, the lower bounds on the spectral gap can also lower bound the exponential convergence rate in the Wasserstein distance. The current sufficient conditions for a spectral gap using the Wasserstein distance are extended to an equivalence condition \citep[Proposition 2.8]{Hairer2014}.
Note that the assumptions in the following proposition are self-contained and hold for reversible Markov kernels under more general conditions.

\begin{proposition}\label{thm:equiv}
Let $\Pi$ be any probability measure on a metric space $(\Omega, d, \B(\Omega))$, and let $P$ be any Markov kernel on $(\Omega, d, \B(\Omega))$ reversible with respect to $\Pi$.
For a $\rho \in (0, 1)$, the following are equivalent if they hold for every probability measure $\mu$ on $\Omega$ with $d\mu / d\Pi \in L^2(\Pi)$:

\noindent (i) The Markov operator $P$ has a spectral gap at least $1 - \rho$.

\noindent (ii) There is a $C_\mu \in (0, \infty)$ depending on $\mu$ such that for every $t \in \Z_+$,
\[
\W_{d \wedge 1}\left( \mu P^t, \Pi \right)
\le C_{\mu} \rho^t.
\]
\end{proposition}

For reversible Metropolis-Hastings satisfying the conditions of Proposition~\ref{prop:conductance_accept_ub}, if Proposition~\ref{thm:equiv} (ii) holds, then
\[
1 - \rho 
\le \inf_{\theta \in \Theta} A(\theta).
\]
Under these conditions, an application of Hölder's inequality extends this lower bound to Wasserstein distances of order $p$.
For probability measures $\mu$ with $d\mu / d\Pi \in L^2(\Pi)$, if there is a $C_\mu \in (0, \infty)$ such that for every $t \in \Z_+$,  
\[
\W_{d}^p\left( \mu P^t, \Pi \right)
\le C_{\mu} \rho^t,
\]
then $\rho \ge 1 - \inf_{\theta} A(\theta)$ as well.

\section{Applications using Metropolis-Hastings}
\label{sec:applications}

\subsection{RWMH with log-concave targets}
\label{sec:RWM_sc}
Let $f : \Theta \to \R$ and $\pi = Z^{-1} \exp(-f)$ where
$Z$ is the normalizing constant.  Consider sampling from $\pi$ with
RWMH using a $d$-dimensional Gaussian proposal centered at the current
state with variance $h I_d$.  In this case, for all
$\theta \in \Theta$,
\[
A(\theta) = \int \left[ \frac{\pi(\theta')}{\pi(\theta)} \wedge 1 \right] q(\theta, \theta') d \theta' \le \frac{1}{(2\pi h)^{d/2}}\int e^{\left( f(\theta) - f(\theta') \right)}  e^{- \frac{1}{2h} \norm{\theta' - \theta}^{2}_{2}} d \theta'.
\]
 Consider the setting where
$f$ satisfies a strong-convexity requirement.  That is, if $S$ is a convex
set and $\xi \in (0, \infty)$, the function $f$ is $\xi^{-1}$-strongly
convex if $\theta \mapsto f(\theta) - \norm{\theta}_2^2 /2\xi$ is convex on $S$.  Also,
let $\partial f(\theta)$ denote the set of subgradients of $f$ at the
point $\theta$.

\begin{proposition}
 \label{prop:sc_rwm}
 Suppose $\Theta$ is convex and that, for
 $\xi \in (0, \infty)$, $f$ is $\xi^{-1}$-strongly convex on
 $\Theta$.  If $\theta_0 \in \Theta$, then, for any
 $v \in \partial f(\theta_{0})$,
 \[
   A(\theta_{0}) \le \frac{1}{(1 + h/\xi )^{d/2}}
   \exp\left\{\frac{h}{2} \left( \frac{1}{1 +
           h/\xi}\right) \|v\|^2_2  \right\}.
 \]
\end{proposition}

There are at least two settings where Proposition~\ref{prop:sc_rwm} yields specific lower bounds.

\begin{corollary}
\label{cor:rwm_max}
  Suppose $\Theta$ is convex and that, for
  $\xi \in (0, \infty)$, $f$ is $\xi^{-1}$-strongly convex on
  $\Theta$. If $\theta^*$ is the point which maximizes $\pi$,
  then 
\[
A(\theta^*) \le (1 + h/\xi )^{-d/2}
\]
 and hence
 \[
\norm{ P^t(\theta^*, \cdot) - \Pi }_{\text{TV}} \ge \left( 1 - \frac{1}{(1 + h/\xi)^{d/2}} \right)^t .
 \]
\end{corollary}

\begin{proof}
This is immediate from Proposition~\ref{prop:sc_rwm} with the observation that $0 \in \partial f (\theta^*)$.
\end{proof}

\begin{corollary}
\label{cor:sc_rwm}
Suppose $\Theta$ is convex and that, for $\xi \in (0, \infty)$, $f$ is $\xi^{-1}$-strongly convex on $\Theta$. If the RWMH kernel is $(\rho,M)$-geometrically ergodic, then
\[
\rho \ge 1 - (1 + h/\xi )^{-d/2}.
\]
\end{corollary}

\begin{proof}
This is immediate from Theorem~\ref{thm:lb_rate} and Corollary~\ref{cor:rwm_max}.
\end{proof}

For any fixed value of $h$, the lower bound will increase exponentially as $d \to \infty$.  However, if $h \propto 1/d$, then this can be avoided, an observation that agrees with the optimal scaling guidelines \citep{robe:weak:1997}, but under much weaker conditions on $\Pi$.

An adversarial example using a Gaussian target showed the spectral gap for the RWMH algorithm tends to $0$ polynomially fast with the dimension \citep{Hairer2014}. An adversarial example in a non-toy example where the convergence rate can tend to $1$ exponentially in the dimension follows.

\begin{example} 
\label{example:RWM}
Suppose, for $i = 1, \ldots, n$, $X_i \in \R^d$ and
\begin{align*}
&\beta \sim N(0, \sigma_{\text{prior}}^2 I_d)
\\
&Y_i|X_i, \beta \stackrel{ind}{\sim} \text{Bern}\left( \left( 1 + \exp\left( -\beta^T X_i \right) \right)^{-1} \right).
\end{align*}
Consider RWMH with Gaussian proposal centered at the current state, scale  $h I_d$, and having the posterior distribution as its invariant distribution.  The RWMH is $(\rho, M)$-geometrically ergodic for some $(\rho, M)$ \citep{Vats2019}. The negative log-likelihood is convex and applying Corollary~\ref{cor:sc_rwm}, obtain
\[
\rho \ge 1 - ( h /\sigma^2_{\text{prior}} + 1 )^{-d/2}.
\]
\end{example}

\subsection{Bayesian logistic regression with Zellner's g-prior}
\label{sec:logistic_Zellner}

A specific application is considered where both $d$ and $n$ are allowed to increase and the data-generating mechanism for Bayesian logistic regression with Zellner's g-prior \citep{Zellner1986} need not be correct. Let $(Y_i, X_i)_{i = 1}^n$ with $Y_i$ taking values in $\{0, 1\}$ and $X_i$ taking values in $\R^d$. Let $g \in (0, \infty)$ be a fixed constant and set $X = (X_1, \ldots, X_n)^T$. 
If $X^T X$ is positive-definite and $s(\cdot)$ denotes the sigmoid function, the posterior density is characterized by
\[
\pi_n(\beta)
\propto \prod_{i = 1}^n s\left( \beta^T X_i \right)^{Y_i} \left( 1 -  s\left( \beta^T X_i \right) \right)^{1 - Y_i} \exp\left( - \frac{1}{2 g} \beta^T X^T X \beta \right).
\]

Assume $(X_{i, j})_{i, j}$ are independent and identically distributed random variables with zero mean, unit variance, and a finite fourth moment.
Consider RWMH with a Gaussian proposal centered at the current state and scale $h I_d$ and invariant density $\pi_n$.  If $d_n/n \to \gamma \in (0, 1)$ as $n \to \infty$, the following result implies that the convergence rate will exponentially tend to $1$ with $n$ unless $h \le 2 g/[ (1 - \sqrt{\gamma})^2 n d_n)]$.

\begin{proposition}
\label{prop:zellner_example}
Let $\beta_n^*$ denote the point which maximizes $\pi_n$.
If $n \to \infty$ in such a way that $d_n/n \to \gamma \in (0, 1)$, then, with probability 1, for all sufficiently large $n$, the acceptance probability for RWMH satisfies
\[
A(\beta^*_n)
\le \left( \frac{h n (1 - \sqrt{\gamma})^2}{2 g} + 1 \right)^{-d_n/2}.
\]
\end{proposition}

It is useful to empirically investigate the convergence of the RWMH algorithm in this example for different tuning parameters.  The acceptance probability at the posterior maximum $\beta^*_n$ is estimated using standard Monte Carlo with $10^3$ samples.
Since $\pi_n$ is log-concave, $\beta^*_n$ can be estimated efficiently using gradient descent.
In turn, Theorem~\ref{thm:lb} will be used to estimate the lower bound to the convergence rate.
Artificial data $(Y_i, X_i)_i$ will be considered where $X_i \sim \text{Unif}(-1, 1)$, $g = 10$, and the data are generated  with increasing dimensions $d$ and sample sizes $n = 4 d$, specifically, 
\[
(d, n) \in \{(2, 8), (4, 16), (4, 24), (8, 32), (10, 40), (12, 48), (14, 56) \}.
\]

As a guideline, Corollary~\ref{prop:zellner_example} says to choose $h \le 20/d^2$ at least when $n$ is large enough.
This example does not satisfy the required theoretical assumptions for optimal scaling guidelines \citep{robe:weak:1997} and hence is anticipated to perform poorly here. It is compared to optimal scaling with $h = 2.38^2/d$, a fixed variance parameter $h = .6$, and scaling with $h = 1/ (d n)$ according to Corollary~\ref{prop:zellner_example}.
Repeating the simulation $50$ times with randomly generated data, Figure~\ref{figure:rwmsim_a} displays the estimates to these lower bounds using the average within $1$ standard error.
Figure~\ref{figure:rwmsim_b} plots the log acceptance probability at the target's maximum to compare the speed at which the convergence rate tends to 1.
According to the theory, optimal scaling and fixed parameter choices should behave poorly as the dimension increases which corresponds with the simulation results shown in Figure~\ref{figure:rwmsim}.

\begin{figure}[t]
\centering
\begin{subfigure}{.45\linewidth}
  \centering
  \includegraphics[width=\linewidth]{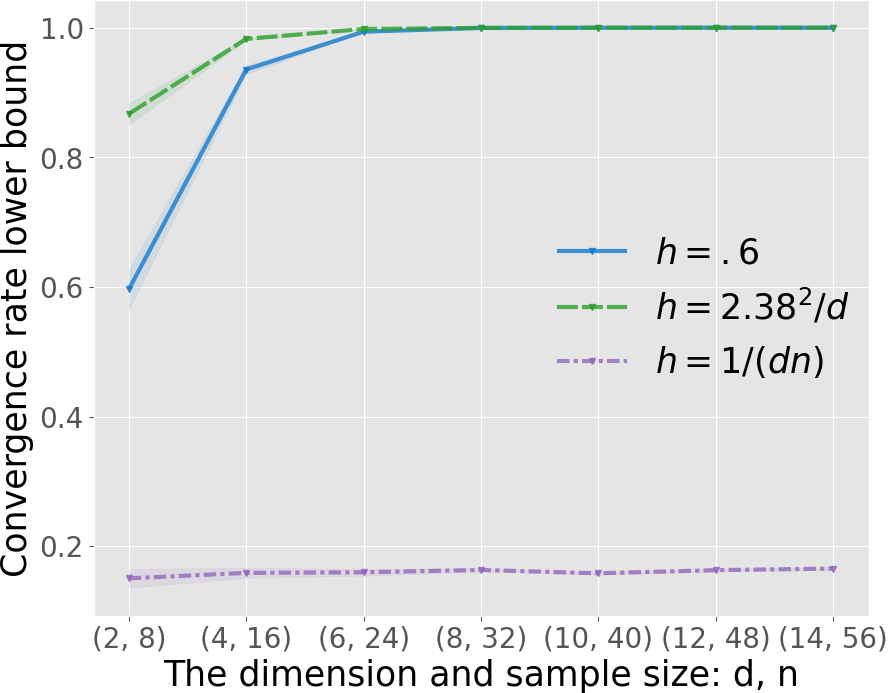}
  \caption{}
  \label{figure:rwmsim_a}
\end{subfigure}
\hspace{.2cm}
\begin{subfigure}{.46\linewidth}
  \centering
  \includegraphics[width=\linewidth]{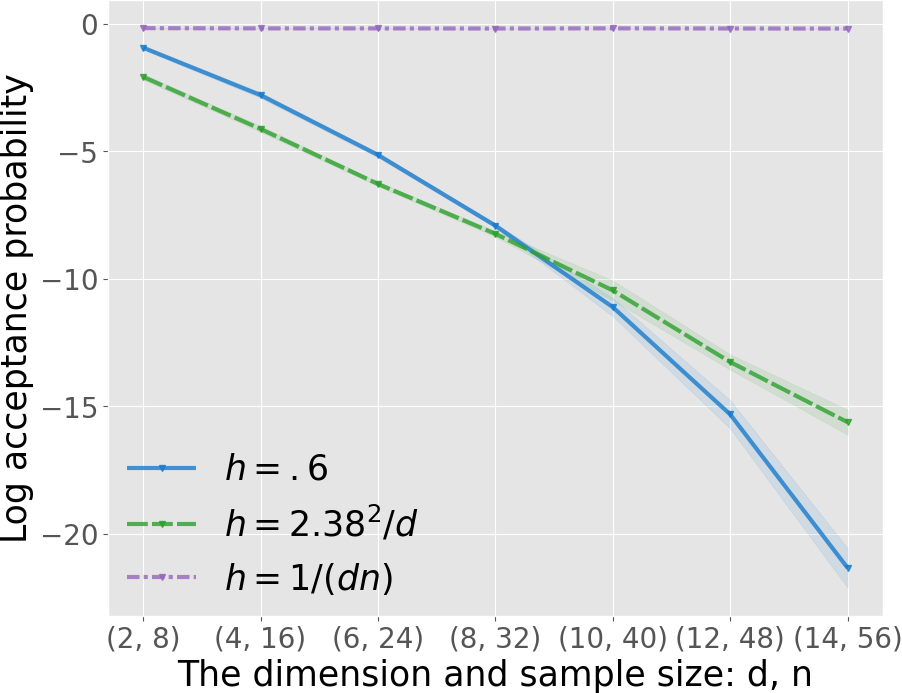}
  \caption{}
  \label{figure:rwmsim_b}
\end{subfigure}
\caption[Lower bounds on the RWMH algorithm for logistic regression with Zellner'g g-prior]{(a) Lower bounds to the rate of convergence of the RWMH algorithm with different scaling values for logistic regression with Zellner's g-prior. (b) Log acceptance probability at the target's maximum. The shaded regions represent $1$ standard deviation from the average after repeated simulations.} \label{figure:rwmsim}
\end{figure}

\subsection{Gaussian proposals for Metropolis-Hastings}
\label{sec:bdd_prop}

Suppose $h \in (0, \infty)$, $C$ is a positive-definite, symmetric matrix, and $\mu : \mathbb{R}^d \to \mathbb{R}^d$. Let the proposal distribution be $N(\mu(\theta), hC)$. For specific choices of $\mu$ and $C$, this is the proposal used in many Metropolis-Hastings algorithms.  For example, if $C=I_d$ and $\mu(\theta) = \theta$, then RWMH results but if $C$ is not the identity, then a Riemannian manifold RWMH algorithm is obtained \citep{Girolami2011}.  Of course, if $\mu(\theta)$ is a constant vector, then an independence sampler results.  
If $\log \pi$ is differentiable, $C=I_d$, and $\mu(\theta) = \theta + h \nabla \log \pi(\theta) / 2$, then this is the proposal used in MALA \citep{Roberts1996Langevin}. 
A general covariance $C$ and $\mu(\theta) = \theta + h C \nabla \log \pi(\theta) / 2$ defines the proposal used in Riemannian manifold MALA \citep{Girolami2011}. 
The following result lower bounds the convergence rate of a Metropolis-Hastings kernel independently of $\mu(\theta)$. 

\begin{proposition}
\label{prop:lb_norm_prop}
Let $\theta_0 \in \Theta$ and the proposal distribution be $N(\mu(\theta), hC)$.  The Metropolis-Hastings acceptance probability  satisfies
\[
A(\theta_0)
\le \frac{1}{\pi(\theta_0) (2 \pi h)^{d/2} \det(C)^{1/2}}.
\]
\end{proposition}

The bound in Proposition~\ref{prop:lb_norm_prop} results in general restrictions on the magnitude of the tuning parameter $h$ in that it forces $h$ to be small in order to avoid poor convergence properties. This suggests, that MH algorithms can have convergence rates with a poor dimension dependence, especially if the tuning parameter is not chosen carefully. An alternative is to choose $h \propto d^{-\delta}$ for some $\delta > 0$, which is reminiscent of the guidelines developed in the optimal scaling
  literature for large $d$, but under much weaker assumptions on the target distribution. In particular, the optimal scaling
  guidelines for RWMH suggest taking $h \propto d^{-1}$
  \citep{robe:weak:1997} while for MALA the optimal scaling guidelines
  suggest taking $h \propto d^{-1/3}$ \citep{robe:rose:1998:optimal}.

Normal proposal distributions are common in many applications of MH, but the key requirement in the proof of Proposition~\ref{prop:lb_norm_prop} is the boundedness of the proposal density.  Consider using a multivariate $t$-distribution as a proposal instead, specifically $t_\nu (\mu(\theta), hC)$.  Then, if $\Gamma(\cdot)$ denotes the usual gamma function,
\[
q(\theta, \theta_0) \le \frac{\Gamma((\nu+d)/2)}{\Gamma(\nu/2) (h \pi)^{d/2} \det(C)^{1/2}}
\]
and the same argument in the proof yields that
\[
A(\theta_0) \le \frac{\Gamma((\nu+d)/2)}{\pi(\theta_0) \Gamma(\nu/2) (h \pi)^{d/2} \det(C)^{1/2}}.
\]
As above, this suggests that large values of $h$ should be avoided.

\subsection{Lower bounds under concentration}
\label{section:concentration}

Section~\ref{sec:logistic_Zellner} considered lower bounds on RWMH under concentration of a strongly log-concave posterior.
If the target density is concentrating to its maximal point with $n$, intuition suggests that the tuning parameters of the Metropolis-Hastings algorithm should also depend on the parameter $n$. 
In infinitely unbalanced Bayesian logistic regression, proposals which depend on the sample size have been shown to exhibit more appealing convergence complexity when compared to data augmentation Gibbs samplers \citep{Johndrow2019}.

Consider target distributions indexed by a parameter $n \in \Z_+$ such as Bayesian posteriors where $n$ is the sample size.  Let $f_n : \R^d \to \R$ and $Z_n =
\int_{\R^d} \exp(-n f_n(\theta)) d\theta$. Define the target density by $\pi_n(\theta) = Z_n^{-1} \exp\left( -n f_n(\theta) \right)$.  If the proposal density is bounded so that there is $B < \infty$, with $B$ possibly depending on $n$ or $d$, such that $q(\theta', \theta) \le B$ (cf. Section~\ref{sec:bdd_prop}), then for MH algorithms, 
\[
A(\theta_0) \le B Z_n e^{n f_n(\theta_0)}.
\]
Of course, the best lower bound will result at $\theta_0 = \theta_n^*$, a point maximizing $\pi_n$.

A Laplace approximation can be used to lower bound the target density at its maximum point. The following result is inspired by previous results on Laplace approximations \citep{Kass1990, Tierney1986, Tierney1989}, but unlike high-dimensional Laplace approximations \citep{Shun1995, Tang2021}, $f_n$ is not required to be differentiable.

\begin{proposition}
\label{prop:concentrate_ub}
Suppose there exists at least one $\theta_n^* \in \Theta$ which maximizes $\pi_n$ for each $n$. Assume for some $\kappa \in (0, 1)$, the dimension $d_n \le n^\kappa$ and that the following holds for some constants $\delta_0, \lambda_0, f^*, I_0 \in (0, \infty)$ and for all sufficiently large $(n, d_n)$:

\noindent 1. $f_n$ is $\lambda_0^{-1}$-strongly convex for all $\norm{\theta -\theta_n^*}_2 \le \delta_0$.

\noindent 2. The optimal point satisfies the strict optimality condition:
\[
\inf_{\norm{v} > \delta_0} f_n(\theta_n^* + v) \ge f_n(\theta_n^*) + f^*.
\]

\noindent 3. The integral is controlled away from the optimum:
\[
\int_{\norm{v}_2 > \delta_0} \exp\left(- \left(f_n(\theta_n^* + v) -  f_n(\theta_n^*) \right) \right) dv \le (I_0 n)^{d_n}.
\]
Then for any $c \in (0, 1]$, for all sufficiently large $(n, d_n)$, the density $\pi_n$ concentrates at $\theta^*_n$ with
\begin{align*}
\pi_n(\theta^*_n)
&\ge \frac{1}{1 + c} \left( \frac{n}{2\pi \lambda_0} \right)^{d_n/2}.
\end{align*}
Moreover, if the proposal density is bounded so that there is $B < \infty$, with $B$ possibly depending on $n$ or $d$, such that $q(\theta', \theta) \le B$, then
\[
A(\theta_n^*) \le B (1 + c)\left( \frac{2 \pi \lambda_0}{n} \right)^{d_n/2}.
\]
\end{proposition}

Proposition~\ref{prop:concentrate_ub} requires the dimension to not grow too fast with $n$.
The first assumption is a locally strong convex assumption which ensures sufficient curvature of $f_n$ locally near the maximum point of $f_n$.
Since only a lower bound on the density is required, the need to control higher order derivatives used in high-dimensional Laplace approximations \citep{Shun1995, Tang2021} is avoided.
The second and third assumptions ensure sufficient decay of $f_n$ away from $\theta^*_n$ similar to assumptions made previously \citep{Kass1990, Tang2021}.
Similar assumptions are also used for Bayesian posterior densities with proper priors when the dimension is fixed \citep[Theorem 4]{Miller2021}. 

\begin{example}
\label{ex:norm_ub}
Recall the definitions of $\mu(\theta)$, $C$, and $h$ from Section~\ref{sec:bdd_prop} and consider MH with a $N(\mu(\theta), hC)$ proposal. Under the conditions of Proposition~\ref{prop:concentrate_ub}, the acceptance probability satisfies
\begin{equation}
\label{eq:high_dim_norm}
A (\theta_n^*)
\le \left( \frac{\lambda_0}{n h} \right)^{d_n/2} \frac{1 + c}{\det(C)^{1/2}}.
\end{equation}
Similarly, Proposition~\ref{prop:concentrate_ub} can be applied with the $t_\nu (\mu(\theta), hC)$ proposal considered in Section~\ref{sec:bdd_prop}.
\end{example}

It is evident from \eqref{eq:high_dim_norm} that  Proposition~\ref{prop:concentrate_ub} can be important when tuning Metropolis-Hastings algorithms used in Bayesian statistics.  In particular, if the target density is concentrating at its maximum point and the tuning parameters $h$ and $C$ do not depend carefully on $n$, then it easily can happen that $\lim_{(n, d_n) \to \infty} A (\theta_n^*) = 0$ rapidly.
Moreover, it can happen that the geometric convergence rate $\lim_{(n, d_n) \to \infty} \rho_n = 1$ rapidly.

\subsection{Flat prior Bayesian logistic regression}

Consider flat prior Bayesian logistic regression without assuming correctness of the data generation.
Let $(Y_i, X_i)_{i = 1}^n$ be independent and identically distributed with $Y_i \in \{0, 1\}$ and $X_i \in \R^d$.   With the sigmoid function $s(\cdot)$, the posterior density is characterized by
\[
\pi_n(\beta)
\propto \prod_{i = 1}^n s\left( \beta^T X_i \right)^{Y_i} \left( 1 -  s\left( \beta^T X_i \right) \right)^{1 - Y_i}.
\]

\begin{assumption}
\label{assumption:logistic_post_exists}
Let $z_i = 1$ if $Y_i = 0$ and $z_i = -1$ if $Y_i = 1$ and define $X^*$ to be the matrix with rows $z_i X_i^T$.
Suppose:

1. $X = (X_1, \ldots, X_n)^T$ is full column rank.

2. There exists a vector $a \in \R^n$ with all components positive such that ${X^*}^T a = 0$.
\end{assumption}

If Assumption~\ref{assumption:logistic_post_exists} holds with probability $1$ for all sufficiently large $n$, then both the maximum likelihood estimator (MLE) $\beta_n^*$ and the random Bayesian logistic regression posterior density, $\pi_n(\beta)$, exist \citep[Theorems 2.1, 3.1]{Chen2000}.  

The Pólya-Gamma Gibbs sampler has been shown to be geometrically ergodic for this model \citep{Wang2018}, but it remains an open question if Metropolis-Hastings is geometrically ergodic.
Since the prior is improper, previous results on posterior concentration \citep{Miller2021} do not apply. However, the next result shows that the posterior density can indeed concentrate so that Proposition~\ref{prop:concentrate_ub} can be applied for MH algorithms with bounded proposals.  For the sake of specificity, recall the definitions of $\mu(\theta)$, $C$, and $h$ from Section~\ref{sec:bdd_prop} so that the MH algorithms uses a $N(\mu(\theta), hC)$ proposal distribution.

\begin{theorem}
\label{thm:flat_lb}
Assume the following:
\begin{enumerate}

\item With probability $1$, Assumption~\ref{assumption:logistic_post_exists} holds for all sufficiently large $n$. \label{assumption:flat:existence}

\item The MLE $\beta^*_n$ is almost surely consistent to some $\beta_0 \in \R^d$. \label{assumption:flat:consistency}

\item $\norm{X_1}_2 \le 1$ with probability $1$. \label{assumption:flat:standardization}

\item For $u \in \R^d$, if $u \not=  0$, then $X_1^T u \not= 0$ with probability $1$. \label{assumption:flat:identifiable}

\end{enumerate}

\noindent Then, with probability $1$, for all sufficiently large $n$, there is a $\lambda_0 \in (0, \infty)$ so that  the MH acceptance probability satisfies
\[
A (\beta_n^*)
\le \left( \frac{\lambda_0}{n h} \right)^{d/2} \frac{2}{\det(C)^{1/2}}.
\]
\end{theorem}

The first assumption ensures existence of the posterior density and MLE \citep[Theorem 2.1, 3.1]{Chen2000}.  Consistency of the MLE is a well-studied problem and conditions are available when the model is correctly specified \citep{Fahrmeir1985} or using M-estimation \citep[Example 5.40]{Vaart1998}. This requires standardization of the features which is often done for numerical stability. The fourth assumption was used previously \citep[Theorem 13]{Miller2021} and is used to ensure identifiability in generalized linear models \citep{Vaart1998}.

In this example an explicit value for $\lambda_0$ is unavailable, but the robustness of the scaling can be investigated empirically via a standard Monte Carlo estimate of the acceptance probability at the MLE $\beta^*_n$.  The estimate will be based on $10^3$ Monte Carlo samples. Consider using MH with a fixed variance parameter $h = .1$, and scaling with $h = 5/n$, $h = 1/n$, and $h = .1/n$. Artificial data is generated with $(Y_i, X_i)_i$ where $X_i \sim \text{Unif}(-1, 1)$ in fixed dimension $d = 10$ and increasing sample sizes $n \in \{ 100, 200, 300, 400 \}$. 
The simulation is replicated $50$ times independently.  
Figure~\ref{figure:flat_rwmsim_a} shows the total convergence rate lower bound and Figure~\ref{figure:flat_rwmsim_b} shows the log acceptance probability at the target's maximum using the average within $1$ estimated standard error.
It is apparent that $h = 5/n$ scales worse than $h = 1/n, .1/n$, but the scaling is not nearly as problematic as with the fixed variance parameter.

\begin{figure}[t]
\centering
\begin{subfigure}{.46\linewidth}
  \centering
  \includegraphics[width=\linewidth]{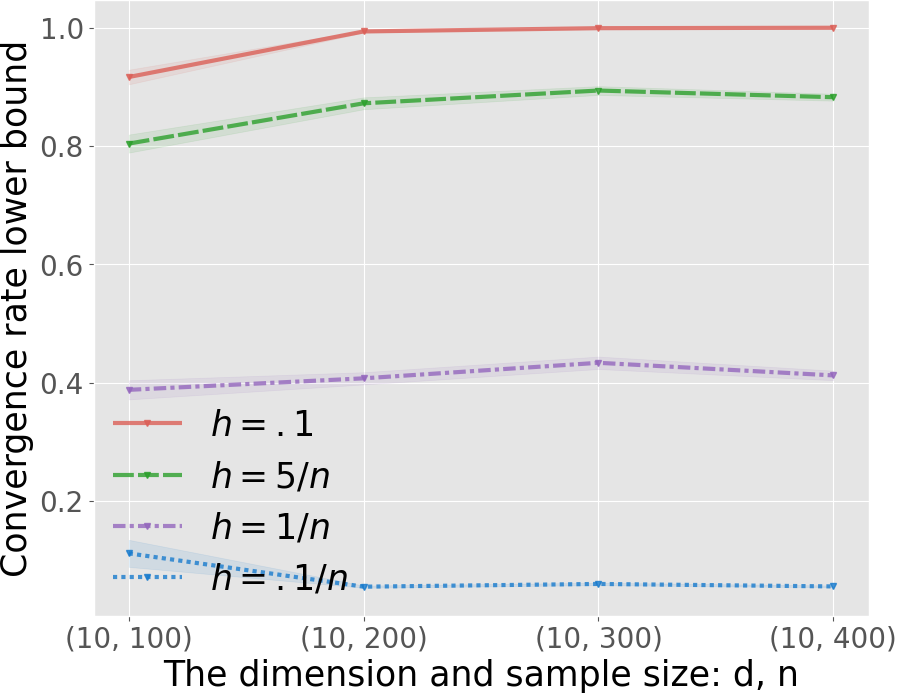}
  \caption{}
  \label{figure:flat_rwmsim_a}
\end{subfigure}
\hspace{.2cm}
\begin{subfigure}{.46\linewidth}
  \centering
  \includegraphics[width=\linewidth]{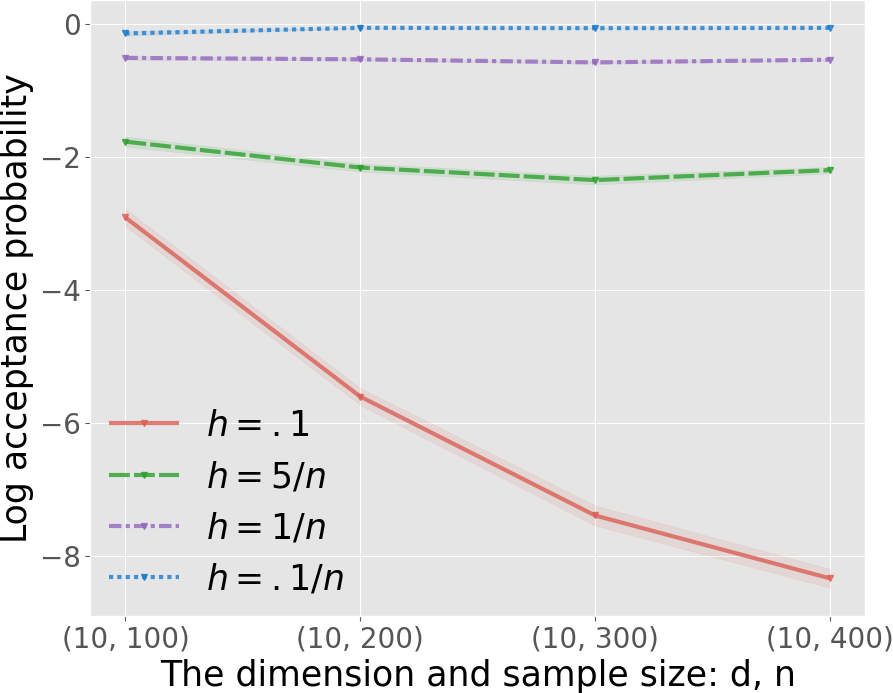}
  \caption{}
  \label{figure:flat_rwmsim_b}
\end{subfigure}
\caption[Lower bounds on the RWMH algorithm for flat prior logistic regression]{(a) Lower bounds to the rate of convergence for flat prior Bayesian logistic regression with RWMH and different scaling values. (b) Log acceptance probability at the target's maximum. The shaded regions represent $1$ standard deviation from the average after repeated simulations.} \label{figure:flat_rwmsim}
\end{figure}

\section{Componentwise Random Scan ARB Markov Chains}
\label{sec:componentwise}

It can be difficult to construct effective ARB chains in high dimensions or when $\Pi$ is complicated.  It is natural to consider so-called componentwise algorithms which consist of ARB Markov kernels that target the full conditionals of $\Pi$.  These algorithms can be useful in applications \citep[e.g.,][]{beze:2018, herb:mcke:2009}, indeed, the introduction of the Metropolis algorithm was a componentwise algorithm \cite{metr:1953}.  The componentwise updates can be combined in many ways through mixing or composition \cite{john:jone:neat:2013}. However, these algorithms are typically more complicated to analyze and their convergence properties have received limited attention \citep[see, e.g.,][]{fort:etal:2003, john:jone:neat:2013, jone:robe:rose:2014, qin:jone:2022, robe:rose:1997}.  

This section considers lower bounds for a componentwise random scan ARB Markov chain, which is now described. Suppose $\Omega = \prod_{k = 1}^M \Omega^k$ and that $\Pi$ continues to be supported on a nonempty measurable set $\Theta = \prod_{k=1}^{M} \Theta^k \subseteq \Omega$.
For each $k = 1, \ldots, M$, define the set $\theta^{(-k)}$ to be $\{ \theta^1, \ldots, \theta^M \} \setminus \{\theta^k \}$ and write the full conditional distributions as $\Pi_{\theta^{(-k)}}$ and densities as $\theta^k \mapsto \pi_{\theta^{(-k)}}(\theta^k)$.
The random scan ARB (ARB-RS) Markov chain is a convex combination of ARB Markov kernels targeting these conditional distributions.
Define $P_{\theta^{(-k)}}$ as the accept-reject-based Markov kernel targeting the conditional $\Pi_{\theta^{(-k)}}$ with proposal $Q_{\theta^{(-k)}}$ where the proposal may depend on $\theta^{(-k)}$.
The conditional distribution may be used as a proposal if it can be sampled directly resulting in a Gibbs update,  but it is assumed below that at least one of the updates is an ARB kernel that is not a Gibbs update.
Denote the corresponding acceptance probability by $A_{\theta^{(-k)}}(\cdot)$.
Let $\lambda_k \in (0, 1)$ be the selection probability of the $k$th component and assume $\sum_{k = 1}^M \lambda_k = 1$.
If $\theta = (\theta^1, \ldots, \theta^M) \in \Omega$ and $B = \prod_{k = 1}^M B^k$ where $B^k \subseteq \Omega^k$, the ARB-RS Markov chain has Markov kernel
\[
P_{RS}(\theta, B)
= \sum_{k = 1}^M \lambda_k \left\{ P_{\theta^{(-k)}}(\theta^k, B^k) \prod_{\theta_j \in \theta^{(-k)}} \delta_{\theta^j}(B^j) \right\}.
\]

\begin{theorem}
\label{thm:lb_rs}
For any $\theta \in \Theta$,
\[
\norm{
P_{RS}^t(\theta, \cdot) - \Pi
}_{\text{TV}}
\ge \left[ 
1 - \sum_{k = 1}^M \lambda_k A_{\theta^{(-k)}}(\theta^k)
\right]^t.
\]
\end{theorem}
This result can be understood as an extension of Theorem~\ref{thm:lb} if when $M=1$, $\theta^{(-k)}$ is understood as null so that $A_{\theta^{(-1)}}(\theta^1) = A(\theta)$ from Section~\ref{sec:arb_mc}.  As in Theorem~\ref{thm:lb_rate}, this result will yield a lower bound on the geometric rate of convergence for $P_{RS}$.
However, the lower bound does not necessarily tend to unity if only one of the componentwise acceptance probabilities in the convex combination tends to unity. Consider the following simple, illustrative example.

\begin{example}
    \label{ex:hybrid}
    Consider the sub-exponential density 
    \[
    \pi(\theta_1, \theta_2) \propto \exp\left\{ -\left(\theta_1^2 + \theta_1^2 \theta_2^2 + \theta_2^2 \right)\right\},  \hspace{5mm} (\theta_1, \theta_2) \in \mathbb{R}^2.
    \]
    Full-dimensional random walk Metropolis-Hastings Markov chains will not be geometrically ergodic \citep{Jarner2000, Roberts1996geo}, but the random scan random walk Metropolis Markov chain is geometrically ergodic \cite{fort:etal:2003}.  The two full conditionals are $\theta_1 \mid \theta_2 \sim \text{N}(0, 1/(2 (1 + \theta_2^2)))$ and  $\theta_2 \mid \theta_1 \sim \text{N}(0, 1/(2 (1 + \theta_1^2)))$.  Consider using a random walk componentwise random scan Metropolis-Hastings algorithm with proposals $\text{N} (\theta_1,  h)$ and $\text{N} (\theta_2,  h)$ for the conditionals $\pi_{\theta_2}(\theta_1)$ and  $\pi_{\theta_1}(\theta_2)$, respectively.  A routine calculation yields
    \[
    A_{\theta_2} (\theta_1) = \int_{\mathbb{R}} \left( 1 \wedge \frac{\pi_{\theta_2}(\theta_1') q_{\theta_2}(\theta_1', \theta_1)}{\pi_{\theta_2}(\theta_1) q_{\theta_2}(\theta_1, \theta_1')} \right)
    q_{\theta_2}(\theta_1, \theta_1')d\theta_1' \le \frac{1}{\sqrt{2h}}
    \]
    and similarly for $A_{\theta_1} (\theta_2)$. If $h \ge 1/2$ and the initial value is the origin, then,  for any random scan selection probabilities,
    \[
    \|P_{RS}^t(0, \cdot) - \Pi\|_{\text{TV}} \ge \left(1- \frac{1}{\sqrt{2h}}\right)^t .
    \]
\end{example}

Under the assumptions of Section~\ref{sec:wass_lb}, there are similar results in many Wasserstein distances

\begin{theorem}
\label{thm:rs_was_lb}
Let assumptions \ref{assumption:ub_pi} and \ref{assumption:wass_condition} hold. If $C_{d, \pi} = C_{0, d} d [ 2 s^{\frac{1}{d}}
  ( 1 + d )^{1 + \frac{1}{d}} ]^{-1}$, then for every $t \in \Z_+$ and every $\theta \in \Theta$, then
\[
\W_{c}^p(P_{RS}^t(\theta, \cdot), \Pi)
\ge C_{d, \pi} \left[ 1 - \sum_{k = 1}^M \lambda_k A_{\theta^{(-k)}}(\theta^k) \right]^{t \left( 1 + \frac{1}{d} \right)}.
\]
\end{theorem}

\section{Final remarks}
\label{sec:final}

The lower bounds developed above are primarily useful for determining when the simulation of an ARB algorithm will fail, even when a good initial value is chosen and the Markov chain is geometrically ergodic.  Optimizing the lower bounds (making them smaller) will not ensure rapid convergence. Particular attention was paid to the convergence complexity of ARB chains and it was demonstrated that the tuning parameters must carefully take into account both the sample size $n$ and parameter dimension $d$. In particular, RWMH and MALA-type algorithms have strong restrictions on the scaling parameter $h$ when the target density concentrates with $n$. If these restrictions are ignored these algorithms will fail. Tuning MH algorithms in high-dimensional, large sample size settings is delicate.

Lower bounds on the convergence rate appear to be available for other sampling algorithms such as  Hamiltonian Monte Carlo, deterministic scan componentwise MCMC (e.g. Metropolis-within-Gibbs), and some adaptive MCMC algorithms, but these are beyond the scope of the current work.  

It is natural to consider if the lower bounds for ARB algorithms have implications for their unadjusted counterparts, such as the unadjusted Langevin algorithm \cite{Roberts1996Langevin}.  The rejection step of an ARB chain is fundamental to the current work implying that alternative methods may be required to develop useful lower bounds.
However, the equivalence in Proposition~\ref{thm:equiv} may be applied to continuous-time Langevin Markov processes and may be beneficial for developing new techniques for their unadjusted counterparts.

\section{Acknowledgments}
The authors thank Riddhiman Bhattacharya, Qian Qin, and Dootika Vats for their helpful comments on an earlier draft. The authors would also like to thank the Associate Editor and two anonymous referees for their constructive comments that helped to improve the paper.

\section{Funding}
Jones was partially supported by NSF grant DMS-2152746.

\section{Supplementary material}
The Python package ``mhlb'' and the code used for the simulations and plots are made available for download at
\url{https://github.com/austindavidbrown/lower-bounds-for-Metropolis-Hastings}.

\begin{appendix}

\section{Proofs for Section~\ref{sec:tv_lb}}
\label{proof:sec:tv_lb}
\begin{proof}[Proof of Theorem~\ref{thm:lb}]
Fix $\theta \in \Theta$ and let $\phi(\cdot) = I_{\{ \theta \}}(\cdot)$.
For any function $\psi : \Theta \to [0, 1]$ any any $\w \in \Theta$,
\[
\int_{\Theta} \psi dP(\w, \cdot)
= \int_{\Theta} \psi a(\w, \cdot) dQ(\w, \cdot) + [1 - A(\w)] \psi(\w)
\ge [1 - A(\w)] \psi(\w).
\]  
Applying this recursively for $t \ge 1$,
\[
\int_{\Theta} \phi dP^t (\theta, \cdot) = \int_{\Theta} \int_{\Theta} \phi dP^{t-1}(\theta, u) dP(u, \cdot) \ge [1 - A(\theta)]^t \phi (\theta).
\]
and since $\pi$ has a density on $\Theta$ with respect to $\lambda$, $\int_{\Theta} \phi d\Pi = 0$.

Let $\Psi$ be the set of functions $\psi : \Theta \to [0, 1]$. Notice that $\phi \in \Psi$ and, since $\pi$ is a density on $\Theta$ with respect to $\lambda$, $\int_{\Theta} \phi d\Pi = 0$.  Thus,
\begin{align*}
\norm{P^t(\theta, \cdot) - \Pi}_{\text{TV}}
&= \sup_{\psi \in \Psi} \left| \int_{\Theta} \psi dP^t(\theta, \cdot) - \int_{\Theta} \psi d\Pi \right|
\\
&\ge \left| \int_{\Theta} \phi dP^t(\theta, \cdot) - \int_{\Theta} \phi d\Pi \right|
\\
&= \int_{\Theta} \phi dP^t(\theta, \cdot)
\\
&\ge [1 - A(\theta)]^t.
\end{align*}
\end{proof}

\begin{proof}[Proof of Theorem~\ref{thm:lb_rate}] 
  Fix $\theta \in \Theta$.  Apply Theorem~\ref{thm:lb} and
  use the assumed geometric ergodicity to obtain for every
  $t \in \Z_+$,
\[
M(\theta) \rho^t
\ge \norm{P^t(\theta, \cdot) - \Pi}_{\text{TV}} 
\ge [1 - A(\theta)]^t
\]
and thus
\[
M(\theta)^{1/t} \rho
\ge 1 - A(\theta).
\]
This implies
\[
\rho
= \lim_{t \to \infty} M(\theta)^{1/t} \rho
\ge 1 - A(\theta).
\]
\end{proof}

\begin{proof}[Proof of Proposition~\ref{prop:conductance_accept_ub}]
Fix $\theta_0 \in \Theta$.
By regularity of $\Pi$ \citep[Lemma 1.36]{Kallenberg2021}, there is a sequence $B_{r_n}(\theta_0) = \{x \in \Theta : d(x, \theta_0) < r_n \} \subseteq \Theta$ of open balls with radius $r_n$ centered at $\theta_0$ such that
\[
\lim_n \Pi(B_{r_n}(\theta_0))
= \Pi(\{ \theta_0 \})
= 0
\]
and $n$ can be chosen large enough so that $\Pi(B_{r_n}(\theta_0)) < 1$.
Since $\theta_0 \in \Theta$ and $B_{r_n}(\theta_0) \subseteq \Theta$, then by assumption on the support of $\pi$ and since $\lambda$ is strictly positive, then 
\[
\Pi(B_{r_n}(\theta_0))
= \int_{B_{r_n}(\theta_0)} \pi d\lambda
> 0.
\]
For $x \in B_{r_n}(\theta_0)$, $\delta_x(B_{r_n}(\theta_0)^c) = 0$ and so,
\begin{align}
k_P
\le \frac{\int_{B_r} P(\cdot, B_{r_n}(\theta_0)^c) d\Pi}{\Pi(B_{r_n}(\theta_0)) \Pi(B_{r_n}(\theta_0)^c)} \nonumber
&= \frac{\int_{B_{r_n}(\theta_0)} \int_{B_{r_n}(\theta_0)^c} a(\cdot, \cdot) dQ(\cdot, \cdot) d\Pi}{\Pi(B_{r_n}(\theta_0)) \Pi(B_{r_n}(\theta_0)^c)} \nonumber
\\
&\le \frac{\int_{B_{r_n}(\theta_0)} A(\cdot) d\Pi}{\Pi(B_{r_n}(\theta_0)) \Pi(B_{r_n}(\theta_0)^c)} \nonumber
\\
&\le \frac{\sup_{x \in B_{r_n}(\theta_0)} A(x)}{\Pi(B_{r_n}(\theta_0)^c)}. \label{eq:conductance_ub}
\end{align}
Choose a sequence $x^*_{n} \in B_{r_n}(\theta_0)$ such that 
\[
\sup_{x \in B_{r_n}(\theta_0)} A(x)
\le A(x^*_{n}) + 1/n.
\]
Since $d(x^*_{n}, \theta_0) < r_n$, then $\lim_{n \to \infty} d(x^*_{n}, \theta_0) = 0$.
By upper semicontinuity of $A(\cdot)$ and taking the limit superior of \eqref{eq:conductance_ub}, 
\[
k_P
\le \limsup_{n \to \infty} \frac{\sup_{x \in B_{r_n}(\theta_0)} A(x)}{1 - \Pi(B_{r_n}(\theta_0))}
\le \limsup_{n \to\infty} A(x^*_{n})
\le A(\theta_0).
\]
Since this holds for every $\theta_0 \in \Theta$, the desired result follows by taking the infimum.
\end{proof}

\section{Proofs for Section~\ref{sec:applications}}

\begin{proof}[Proof of Proposition~\ref{prop:lb_norm_prop}]
For all $\theta'$, $\theta$, 
\[
q(\theta, \theta') \le (2 \pi h)^{-d/2} \det(C)^{-1/2}
\]
and hence
\begin{align}
A(\theta_0) & = \int_{\Theta} \left\{ 1 \wedge \frac{\pi(\theta') q(\theta', \theta_0)}{\pi(\theta_0) q(\theta_0,\theta')} \right\} q(\theta_0, \theta') d \theta'\nonumber \\
& = \int_{\Theta} \left\{ \frac{q(\theta_0, \theta')}{\pi(\theta')} \wedge \frac{q(\theta', \theta_0)}{\pi(\theta_0) } \right\} \pi(\theta') d \theta'\nonumber \\
& \le \int_{\Theta} \frac{q(\theta', \theta_0)}{\pi(\theta_0)}  \pi(\theta') d \theta'\nonumber \\
&\le \frac{1}{\pi(\theta_0) (2 \pi h)^{d/2} \det(C)^{1/2}}. \label{inequality:compare_accepts_rwm}
\end{align}
\end{proof}

\begin{proof}[Proof of Proposition~\ref{prop:sc_rwm}]
  By the subgradient inequality
  \citep[Corollary 3.2.1]{Nesterov2018}, for every
  $\theta, \theta_0 \in \Theta$ and $v \in \partial f(\theta_{0})$,
\begin{align*}
f(\theta) - f(\theta_0) 
&\ge v^T(\theta - \theta_0) + \frac{1}{2 \xi} \norm{\theta - \theta_0}_2^2.
\end{align*}
Thus,
\begin{align*}
  A(\theta_0)
  &\le \frac{1}{(2 \pi h)^{d/2}} \int \exp\left\{ f(\theta_0) - f(\theta) \right\} \exp\left\{ - \frac{1}{2 h} \norm{\theta - \theta_0}_2^2 \right\} d\theta
  \\
  &\le \frac{1}{(2 \pi h)^{d/2}} \int \exp\left\{v^T(\theta_0 -
    \theta) + \frac{1}{2 \xi} \norm{\theta - \theta_0}_2^2\right\} \exp\left\{ - \frac{1}{2 h} \norm{\theta - \theta_0}_2^2
    \right\} d\theta
\end{align*}
and a routine calculation involving completing the square yields the claim.
\end{proof}

\begin{proof}[Proof of Proposition~\ref{prop:zellner_example}]
To apply Proposition~\ref{prop:sc_rwm}, it will be shown that  with probability 1, for sufficiently large $n$, the target density is strongly convex.

Notice that, with probability $1$, if $d_n/n \to \gamma \in (0, 1)$, then
\[
\lambda_{min}\left( \frac{1}{n} X^T X \right)
\ge \frac{1}{2} (1 - \sqrt{\gamma})^2
\]
for all $n$ sufficiently large \citep[Theorem 2]{Bai1993}. Define 
\[
f_n(\beta) = \frac{1}{n} \sum_{i = 1}^n \left[ \log\left( 1 + \exp\left( \beta^T X_i \right) \right) - Y_i \beta^T X_i \right] + \frac{1}{2 g n} \beta^T X^T X \beta,
\] 
and write the posterior density $\pi_n \propto \exp(-n f_n)$. Let ${\nabla}^2 f_n$ denote the Hessian of $f_n$.
For every $v, u \in \R^d$,
\[
u^T {\nabla}^2 f_n(v) u
\ge \frac{1}{g n} u^T X^T X u
\ge \frac{(1 - \sqrt{\gamma})^2}{2 g} \norm{u}_2^2.
\]
Then $n f_n(\cdot)$ is strongly convex with convexity parameter $\frac{n (1 - \sqrt{\gamma})^2}{2 g}$ on $\R^d$ with probability $1$ \citep[Theorem 2.1.11]{Nesterov2018}.
\end{proof}

\begin{proof}[Proof of Proposition~\ref{prop:concentrate_ub}]
Take $n$ sufficiently large so that the each of the assumptions hold. After changing the variables, obtain the decomposition
\begin{align}
\frac{1}{\pi(\theta^*_n)}
&= \int_{\R^d} \exp(-n (f_n(\theta) - f_n(\theta_n^*))) d\theta \nonumber
\\
&= \frac{1}{n^{d/2}} \int_{\norm{u  n^{-1/2}}_2 \le \delta_0} \exp(-n (f_n(\theta_n^* + u n^{-1/2}) - f_n(\theta_n^*))) du \label{eq:decomposition1} 
\\
&\hspace{.4cm}+ \int_{\norm{v}_2 > \delta_0} \exp(-n (f_n(\theta_n^* + v ) - f_n(\theta_n^*))) dv \label{eq:decomposition2}.
\end{align}

Consider the first integral \eqref{eq:decomposition1}. 
Since the closed ball $\overline{B_{\delta_0}(\theta_n^*)}$ is convex, for all $\norm{u n^{-1/2}}_2 \le \delta_0$, by the subgradient inequality \citep[Lemma 3.2.3]{Nesterov2018},
\begin{align*}
n f_n(\theta^*_n + u n^{-1/2}) - n f_n(\theta^*_n)
&\ge \frac{1}{2 \lambda_0} \norm{u}_2^2.
\end{align*}
This implies
\begin{align*}
&\frac{1}{n^{d_n/2}} \int_{\norm{n^{-1/2} u}_2 \le \delta_0} \exp(n f_n(\theta^*_n) - n f_n(\theta^*_n + u n^{-1/2})) du
\\
&\le \frac{1}{n^{d_n/2}}  \int_{\norm{u n^{-1/2}}_2 \le \delta_0} \exp\left( -\frac{1}{2 \lambda_0} \norm{u}_2^2 \right) du
\\
&\le \left( \frac{2\pi \lambda_0}{n} \right)^{d_n/2}.
\end{align*}

Consider the second integral \eqref{eq:decomposition2}
\begin{align*}
&\int_{\norm{v}_2 > \delta_0} \exp\left(-n \left( f_n(\theta_n^* + v) - f_n(\theta_n^*) \right) \right) dv
\\
&= \int_{\norm{v}_2 > \delta_0} \exp\left(-(n - 1) \left( f_n(\theta_n^* + v) - f_n(\theta_n^*) \right) - \left( f_n(\theta_n^* + v) - f_n(\theta_n^*) \right) \right) dv
\\
&\le e^{-(n - 1) f^*} \int_{\norm{v}_2 > \delta_0} \exp\left(-\left( f_n(\theta_n^* + v) - f_n(\theta_n^*) \right) \right) dv
\\
&\le e^{-(n - 1) f^*} (I_0 n)^{d_n}.
\end{align*}
Combining these results,
\begin{align*}
&\int_{\R^d} \exp(-n (f_n(\theta) - f_n(\theta_n^*))) d\theta
\\
&\le \left( \frac{2\pi \lambda_0}{n} \right)^{d_n/2}
\left[1 + \left( \frac{n}{2\pi \lambda_0 } \right)^{d_n/2} e^{-(n - 1) f^*} (I_0 n)^{d_n} \right].
\end{align*}
Since $d_n \le n^{\kappa}$,
\begin{align*}
\limsup_{n \to \infty} \left( \frac{n}{2\pi \lambda_0 } \right)^{d_n/2} e^{-(n - 1) f^*} (I_0 n)^{d_n} = 0.
\end{align*}
The desired result follows at once.
\end{proof}

\begin{proof}[Proof of Theorem~\ref{thm:flat_lb}] 
The proof will show the conditions of Proposition~\ref{prop:concentrate_ub} hold with probability $1$ for large enough $n$.
Using Assumption~\ref{assumption:flat:existence}, with probability $1$, assume $n$ is sufficiently large so that the posterior density $\pi_n^*$ and the MLE $\beta^*_n$ exist.
Define 
\[
f_n(\beta) = \frac{1}{n} \sum_{i = 1}^n \left[ \log\left( 1 + \exp\left( \beta^T X_i \right) \right) - Y_i \beta^T X_i \right]
\] 
and write $\pi_n^* = Z_n^{-1} \exp(-n f_n)$ where $Z_n = \int \exp(-n f_n(\beta)) d\beta$.
The first step is to develop sufficient curvature of the target density at $\beta_n^*$.
Denote the $k$th derivative matrix or tensor of the function $f_n$ by ${\nabla}^{k} f_n$.
Recall that $s$ is the sigmoid function. For every $v \in \R^d$, 
\[
v^T{\nabla}^2 f_n(\beta_n^*) v
= \frac{1}{n} \sum_{i = 1}^n v^T X_i X_i^T v s\left( X_i^T \beta_n^* \right) \left( 1 - s\left( X_i^T \beta_n^* \right) \right)
\]
and
\[
{\nabla}^3 f_n(\beta_n^*) v^{(3)}
= \frac{1}{n} \sum_{i = 1}^n ( X_i^T v )^3 s\left( X_i^T \beta_n^* \right) \left( 1 - s\left( X_i^T \beta_n^* \right) \right) \left( 1 - 2 s\left( X_i^T \beta_n^* \right) \right).
\]
Since $\norm{X_i}_2 \le 1$ with probability $1$, by the strong law of large numbers \citep[Theorem 10.13]{Folland2013}, almost surely,
\begin{align*}
{\nabla}^2 f_n(\beta_0)
&= \frac{1}{n} \sum_{i = 1}^n X_i X_i^T s\left( X_i^T \beta_0 \right) \left( 1 - s\left( X_i^T \beta_0 \right) \right)
\\
&\to \E\left( X_1 X_1^T s\left( X_1^T \beta_0 \right) \left( 1 - s\left( X_1^T \beta_0 \right) \right) \right).
\end{align*}
By Assumption~\ref{assumption:flat:identifiable}, for any $u \in \R^d$, $u \not= 0$, with probability $1$,
\[
u^T X_1 X_1^T u s\left( X_1^T \beta_0 \right) \left( 1 - s\left( X_1^T \beta_0 \right) \right) > 0.
\]
Since expectations preserve strict inequalities, there is a sufficiently small $\e_0 \in (0, 1)$ so that for any $u \in \R^d$, $u \not= 0$,
\[
\left( \frac{u}{\norm{u}_2} \right)^T \E\left( X_1 X_1^T s\left( X_1^T \beta_0 \right) \left( 1 - s\left( X_1^T \beta_0 \right) \right) \right) \left( \frac{u}{\norm{u}_2} \right) \ge \e_0.
\]
Combining these results, for all $u \in \R^d$, with probability $1$,
\begin{align*}
u^T {\nabla}^2 f_n(\beta_0) u
&\ge \frac{\e_0}{2} \norm{u}_2^2
\end{align*}
for all sufficiently large $n$.
By Assumption~\ref{assumption:flat:standardization}, $\norm{X_i}_2 \le 1$ with probability $1$ and $|x(1 - x) (1 - 2x)| \le 1/4$, so by the mean value theorem, almost surely,
\[
\norm{{\nabla}^2 f_n(\beta_n^*) - {\nabla}^2 f_n(\beta_0)}_2
\le 4^{-1} \norm{\beta^*_n - \beta_0}_2
\to 0.
\]
For all $u \in \R^d$, with probability $1$, for all $n$ sufficiently large,
\begin{align}
u^T {\nabla}^2 f_n(\beta_n^*) u
&\ge u^T {\nabla}^2 f_n(\beta_0) u - \frac{\e_0}{4} \norm{u}_2^2 \nonumber
\\
&\ge \frac{\e_0}{4} \norm{u}_2^2. \label{eq:flat:curvature}
\end{align}

For the remainder of this argument, assume $n$ is sufficiently large so that \eqref{eq:flat:curvature} holds with probability $1$ and the remainder of the proof is taken to hold with probability $1$ without reference.
Since $\norm{X_i}_2 \le 1$, for all $u \in \R^d$ and $t \in (0, \infty)$,
\[
| {\nabla}^3 f_n(\beta_n^* + t u) u^{(3)} |
\le \norm{u}_2 u^T{\nabla}^2 f_n(\beta_n^* + t u) u .
\]
It is immediate that \citep[Proposition 1 (6)]{Bach2010}
for $\norm{v}_2 \le 2$ and  all $u, v \in \R^d$,
\[
u^T {\nabla}^2 f_n(\beta_n^* + v) u
\ge u^T {\nabla}^2 f_n(\beta_n^*) u \exp(-\norm{v}_2)
\ge \frac{\e_0}{4} \exp(-2) \norm{u}_2^2.
\]
Since closed balls are convex,  $f_n$ is strongly convex on the closed ball $\overline{B_2(\beta_n^*)}$ \citep[Theorem 2.1.11]{Nesterov2018}.
Thus, the local strong convexity condition (1) in Proposition~\ref{prop:concentrate_ub} holds.

Since $\nabla f_n(\beta_n^*) \equiv 0$ \citep[Proposition 1 (3)]{Bach2010} 
\[
f_n(\beta_n^* + v) - f_n(\beta_n^*)
\ge \frac{v^T {\nabla}^2 f_n(\beta_n^*) v}{\norm{v}_2^2} \left( \exp(-\norm{v}_2) + \norm{v}_2 - 1 \right)
\]
Using \eqref{eq:flat:curvature} and the fact that $e^{-x} \ge 1 - x$, obtain, for all $v \in \R^d$ with $v \not= 0$, 
\begin{align*}
\frac{v^T {\nabla}^2 f_n(\beta_n^*) v}{\norm{v}_2^2} \left( \exp(-\norm{v}_2) + \norm{v}_2 - 1 \right)
&\ge \frac{\e_0}{4} \left( \exp(-\norm{v}_2) + \norm{v}_2 - 1 \right)
\\
&\ge \frac{\e_0}{4} \left(\norm{v}_2 - 1 \right).
\end{align*}
Hence, for all $v \in \R^d$ with $v \not= 0$,
\[
f_n(\beta_n^* + v) - f_n(\beta_n^*) \ge \frac{\e_0}{4} \left(\norm{v}_2 - 1 \right) .
\]
Thus, for $\norm{v}_2 > 2$, the strict optimality condition (2) in Proposition~\ref{prop:concentrate_ub} holds.

The required control of the integral (3) in Proposition~\ref{prop:concentrate_ub} also holds since
\begin{align*}
\int_{\norm{v}_2 > 2} \exp\left(- \left( f_n(\beta_n^* + v) - f_n(\beta_n^*) \right) \right) dv
&\le e^\frac{\e_0}{4} \int_{\norm{v}_2 > 2} \exp\left(- \frac{\e_0}{4} \norm{v}_2 \right) dv
< \infty.
\end{align*}

\end{proof}

\section{Proofs for Section~\ref{sec:wass_lb}}

\begin{proof}[Proof of Theorem~\ref{thm:wasserstein_lb}]
Begin by constructing a suitable Lipschitz function.
Fix $\theta \in \R^d$, and fix $\alpha \in (0, \infty)$.
Define the function $\phi_{\alpha, \theta} : \R^d \to \R$ by $\phi_{\alpha, \theta}(\w) = \exp\left( - \alpha \norm{\w - \theta}_1 \right)$.
Then for every $\w, \w' \in \R^d$,
\begin{align*}
|\phi_{\alpha, \theta}(\w) - \phi_{\alpha, \theta}(\w')|
&= |\exp\left(-\alpha \norm{\w - \theta}_1 \right) - \exp\left(-\alpha \norm{\w' - \theta}_1 \right)|
\\
&\le \alpha|\norm{\w - \theta}_1 - \norm{\w' - \theta}_1|
\\
&\le \alpha \norm{\w - \w'}_1.
\end{align*}
Therefore, $\alpha^{-1} \phi_{\alpha, \theta}$ is a bounded Lipschitz function with respect to the distance $\norm{\cdot}_1$ and the Lipschitz constant is $1$.
By assumption there exists $s$ such that $\pi \le s$. Then, using the fact that $\int_{\R^d} \exp\left( -\alpha \norm{\theta' - \theta}_1 \right) d\theta' = 2^d \alpha^{-d}$, obtain
\begin{align}
\int \phi_{\alpha, \theta}(\theta') \pi(\theta') d\theta' \nonumber
&\le s \int_{\R^d} \exp\left( -\alpha \norm{\theta' - \theta}_1 \right) d\theta' \nonumber
\\
&= s 2^d \alpha^{-d}. \label{eq:approxub}
\end{align}
Fix a positive integer $t$.
Then, for each $s \in \{ 0, \ldots, t - 1 \}$ and each $\w \in \R^d$, obtain the lower bound
\begin{align*}
&\int \phi_{\alpha, \theta}(\theta') \left( 1 - A(\theta') \right)^s P(\w, d\theta')
\\
&= \int \phi_{\alpha, \theta}(\theta') \left( 1 - A(\theta') \right)^s a(\w, \theta') dQ(\w, \theta')
+ \phi_{\alpha, \theta}(\w) \left( 1 - A(\w) \right)^{s + 1}
\\
&\ge \phi_{\alpha, \theta}(\w) \left( 1 - A(\w) \right)^{s + 1}.
\end{align*}

Now apply this lower bound multiple times:
\begin{align}
\int \phi_{\alpha, \theta}(\theta_{t}) P^t(\theta, d\theta_{t}) 
&= \int \left\{ \int \phi_{\alpha, \theta}(\theta_{t}) P(\theta_{t-1}, d\theta_{t}) \right\} P^{t-1}(\theta, d\theta_{t-1}) \nonumber
\\
&\ge \int \phi_{\alpha, \theta}(\theta_{t - 1}) \left( 1 - A(\theta_{t-1}) \right) P^{t-1}(\theta, d\theta_{t-1}) \nonumber
\\
&\vdots \nonumber
\\
&\ge \phi_{\alpha, \theta}(\theta) \left( 1 - A(\theta) \right)^{t} \nonumber
\\
&= \left( 1 - A(\theta) \right)^{t} \label{eq:phiapproxlb}.
\end{align}
The final step follows from the fact that $\phi_{\alpha, \theta}(\theta) = 1$.
Combining \eqref{eq:approxub} and \eqref{eq:phiapproxlb}, obtain the lower bound,
\begin{align}
\int \left( \frac{1}{\alpha} \phi_{\alpha, \theta} \right) dP^t(\theta, \cdot) - \int \left( \frac{1}{\alpha} \phi_{\alpha, \theta} \right) d\Pi 
\ge \frac{ \left( 1 - A(\theta) \right)^t - s 2^d \alpha^{-d}}{\alpha}. \label{eq:approxliplb}
\end{align}

The case where $\W_{\norm{\cdot}_1}(P^t(\theta, \cdot), \Pi) = +\infty$ is trivial so assume this is finite.
Then by the Kantovorich-Rubinshtein theorem \citep[][Theorem 1.14]{Villani2003} and the lower bound in \eqref{eq:approxliplb},
\begin{align*}
\W_{\norm{\cdot}_1}(P^t(\theta, \cdot), \Pi)
&\ge \sup_{ \norm{\phi}_{\text{Lip}(\norm{\cdot}_1)} \le 1} \left[ \int \phi dP^t(\theta, \cdot) - \int \phi d\Pi \right]
\\
&\ge \int \left( \frac{1}{\alpha} \phi_{\alpha, \theta} \right) dP^t(\theta, \cdot) - \int \left( \frac{1}{\alpha} \phi_{\alpha, \theta} \right) d\Pi 
\\
&\ge  \frac{ \left( 1 - A(\theta) \right)^t - s 2^d \alpha^{-d}}{\alpha}.
\end{align*}
If $A(\theta) = 1$, then taking the limit of $\alpha \to \infty$, completes the proof.
Suppose then that $A(\theta) < 1$. Maximizing this lower bound with respect to $\alpha$ yields $\alpha = 2 s^{\frac{1}{d}} (1 + d)^{\frac{1}{d}} \left( 1 - A(\theta) \right)^{-\frac{t}{d}}$.
Then
\begin{align*}
\frac{ \left( 1 - A(\theta) \right)^t - s 2^d \alpha^{-d}}{\alpha}
&= \frac{\left( 1 - \frac{1}{1 + d} \right)}{2 s^{\frac{1}{d}} \left( 1 + d \right)^{\frac{1}{d}}} \left( 1 - A(\theta) \right)^{t \left(1 + \frac{1}{d} \right)}.
\end{align*}
This completes the proof for the norm $\norm{\cdot}_1$ and by assumption, $\W_{c}(\cdot, \cdot) \ge C_{0, d} \W_{\norm{\cdot}_1}(\cdot, \cdot)$.

Finally, let $\Gamma$ be a coupling for $P^t(\theta, \cdot)$ and $\Pi$. 
Using Hölder's inequality \citep[Theorem 6.2]{Folland2013} with $1/p + 1/q = 1$,
\begin{align*}
\W_{c}^1(P^t(\theta, \cdot), \Pi)
&\le \int c(\theta', \w') \cdot 1 d\Gamma(\theta', \w')
\\
&\le \left( \int c(\theta', \w')^p d\Gamma(\theta', \w') \right)^{1/p} \left( \int 1^q d\Gamma(\theta, \theta') \right)^{1/q}
\\
&\le \left( \int c(\theta', \w')^p d\Gamma(\theta', \w') \right)^{1/p}.
\end{align*}
Taking the infimum over $\Gamma$ completes the proof.
\end{proof}

\begin{proof}[Proof of Proposition~\ref{thm:equiv}]
It will suffice to show (ii) implies (i) \cite[Theorem 2.1]{robe:rose:1997}. 
Condition  (ii) implies convergence in the bounded Lipschitz norm, that is, 
\[
\sup\left\{ \int \phi d\mu P^t - \int \phi d\Pi : \norm{\phi}_{\text{Lip}(d)}   + \sup_{x} |\phi(x)| \le 1 \right\}
\le C_{\mu} \rho^t.
\]
The following technique and proof are similar to that of an existing result, but there is additional work required to ensure the initialization can be restricted to obtain the equivalence as well as use the bounded Lipschitz norm on a general metric space \citep[Proposition 2.8]{Hairer2014}. 
Let $t \in \Z_+$ and let $\phi : \Omega \to \R$ be a non-negative, bounded, Lipschitz with Lispchitz constant $L(\phi) \in (0, \infty)$. 
Assume $\phi > 0$ on some positive probability of $\Pi$ so that $\int \phi d\Pi > 0$. 
Define $g = \phi/\int \phi d\Pi$ so that $\int g d\Pi = 1$. The function $g$ is Lipschitz since for $x, y \in \Omega$,
\[
|g(x) - g(y)|
= \frac{|\phi(x) - \phi(y)|}{\int \phi d\Pi}
\le \frac{L(\phi)}{\int \phi d\Pi} d(x, y).
\]

Define the probability measure by $\mu_\phi(B) = \int_B g d\Pi$ and so $\mu_\phi \in L^2(\Pi)$.
Thus $g$ has been constructed so that
\[
\norm{g}_{\text{Lip}(d)}  + \sup_{x} |g(x)| 
\le \frac{L(\phi)}{\int \phi d\Pi} + \frac{\sup_{x} |\phi(x)|}{\int \phi d\Pi}.
\]
By (ii), there is a constant $C_{\mu_\phi} \in (0, \infty)$ such that
\[
\int g d\mu_{\phi} P^t - \int g d\Pi
\le C_{\mu_\phi} \left[ \frac{L(\phi)}{\int \phi d\Pi} + \frac{\sup_{x} |\phi(x)|}{\int \phi d\Pi} \right] \rho^{t}.
\]
By reversibility, $\int (P^n g)^2 d\Pi = \int (P^{2n} g) g d\Pi$.
Then
\begin{align*}
\frac{1}{(\int \phi d\Pi)^2} \norm{P^t \phi - \int \phi d\Pi}_{L^2(\Pi)}^2
&= \norm{P^t g - \int g d\Pi}_{L^2(\Pi)}^2
\\
&= \int (P^t g)^2 d\Pi - \left(\int g d\Pi \right)^2
\\
&= \int (P^{2 t} g) g d\Pi - \int g d\Pi
\\
&= \int g \mu_\phi P^{2 t} - \int g d\Pi
\\
&\le C_{\mu_\phi} \left[ \frac{L(\phi)}{\int \phi d\Pi} + \frac{\sup_{x} |\phi(x)|}{\int \phi d\Pi} \right] \rho^{2 t}.
\end{align*}
Now let $\phi : \Omega \to \R$ be a bounded, Lipschitz function with Lipschitz constant $L(\phi) \in (0, \infty)$.
Denote the positive and negative parts of $\phi$ by $\phi^+, \phi^-$. 
There are then constants $C_{\mu_{\phi^+}}, C_{\mu_{\phi^-}} \in (0, \infty)$ such that
\begin{align*}
\norm{P^t \phi - \int \phi d\Pi}_{L^2(\Pi)}^2
&\le 
2 C_{\mu_{\phi^+}} \int \phi^+ d\Pi \left[ L(\phi) + \sup_{x} |\phi^+(x)| \right] \rho^{2 t}
\\
&\hspace{.4cm}+ 2 C_{\mu_{\phi^-}} \int \phi^- d\Pi \left[ L(\phi) + \sup_{x} |\phi^-(x)| \right] \rho^{2 t}.
\end{align*}
This holds for arbitrary $t$ and so applying \citep[Lemma 2.9]{Hairer2014},
\[
\norm{P^t \phi - \int \phi d\Pi}_{L^2(\Pi)}^2
\le \rho^{2t} \norm{\phi - \int \phi d\Pi}_{L^2(\Pi)}^2.
\]

Now let $\phi : \Omega \to \R$ be a bounded, lower semicontinuous function with $\int \phi d\Pi = 0$.
The approximation $\phi_m(\cdot) = \inf_{y}[ \phi(y) + m d(\cdot, y) ]$ is Lipschitz continuous, uniformly bounded and $\phi_m \uparrow \phi$ pointwise.
Now by the dominated convergence theorem \cite[Theorem 2.24]{Folland2013}, $P \phi_m \uparrow P \phi$ pointwise.
By the dominated convergence theorem and since $\phi_m$ is bounded and Lipschitz, 
\begin{align*}
\norm{P \phi}_{L^2(\Pi)}^2
&= \lim_{m \to \infty} \norm{P \phi_m}_{L^2(\Pi)}^2
\\
&\le \rho^{2} \lim_{m \to \infty} \norm{\phi_m}_{L^2(\Pi)}^2
\\
&= \rho^{2} \norm{\phi}_{L^2(\Pi)}^2.
\end{align*}

Now let $\phi \in L^2(\Pi)$ with $\int \phi d\Pi = 0$. 
By \citep[Lemma 1.37]{Kallenberg2021}, choose a sequence $(\phi_m)_m$ of bounded, continuous functions converging to $\phi$ in $L^2(\Pi)$.
By reversibility and Jensen's inequality, 
\begin{align*}
\int (P\phi_m - P\phi)^2 d\Pi
&\le \int \int (\phi_m(y) - \phi(y))^2 dP(x, y) d\Pi(x)
\\
&= \int \int (\phi_m(y) - \phi(y))^2 d\Pi(y) dP(y, x)
\\
&\to 0.
\end{align*}
So, if $\phi_m \to \phi$ in $L^2(\Pi)$, then $P\phi_m \to P\phi$ in $L^2(\Pi)$ as well.
Therefore, since $\phi_m$ is bounded and lower semicontinuous, after taking limits, conclude
\begin{align*}
\norm{P \phi}_{L^2(\Pi)}^2
&\le \rho^{2} \norm{\phi}_{L^2(\Pi)}^2.
\end{align*}
The condition $\int \phi d\Pi = 0$ can be extended to all of $L^2(\Pi)$ by shifting $\phi$ and thus this holds for all of $L^2(\Pi)$.
\end{proof}

\section{Proofs for Section~\ref{sec:componentwise}}
\label{proof:sec:rs}

\begin{proof}[Proof for Theorem~\ref{thm:lb_rs}]
Fix $\theta \in \Theta$.
For any function $\psi : \Theta \to [0, 1]$ any any $\w \in \Theta$,
\[
\int_{\Theta} \psi dP(\w, \cdot)
\ge [1 - \sum_{k = 1}^M \lambda_k A_{\w^{(-k)}}(\w^k)] \psi(\w).
\]  
Fix $\theta \in \Theta$ and choose the function $\phi(\cdot) = I_{\{ \theta \}}$.
Applying this recursively
\[
\int_{\Theta} \phi dP^t (\theta, \cdot)
\ge \left[ 1 - \sum_{k = 1}^M \lambda_k A_{\theta^{(-k)}}(\theta^k) \right]^t \phi(\theta).
\]
Using the assumption on $\Pi$, this implies the lower bound
\begin{align*}
\norm{P^t(\theta, \cdot) - \Pi}_{\text{TV}}
\ge \int_{\Theta} \phi dP^t(\theta, \cdot)
\ge \left[ 1 - \sum_{k = 1}^M \lambda_k A_{\theta^{(-k)}}(\theta^k) \right]^t.
\end{align*}
\end{proof}

\begin{proof}[Proof of Theorem~\ref{thm:rs_was_lb}]

The proof is similar to Theorem~\ref{thm:wasserstein_lb}.
Fix $\theta \in \R^d$, and fix $\alpha \in (0, \infty)$.
Define the function $\phi$ by $\phi_{\alpha, \theta}(\w) = \exp\left( - \alpha \norm{\w - \theta}_1 \right)$. Similarly obtain the lower bound
\begin{align}
\int \phi_{\alpha, \theta} dP^t(\theta, \cdot) 
&\ge \left[ 1 - \sum_{k = 1}^M \lambda_k A_{\theta^{(-k)}}(\theta^k) \right]^t
\end{align}
using the fact that $\phi_{\alpha, \theta}(\theta) = 1$.
This gives the lower bound
\begin{align}
\int \left( \frac{1}{\alpha} \phi_{\alpha, \theta} \right) dP^t(\theta, \cdot) - \int \left( \frac{1}{\alpha} \phi_{\alpha, \theta} \right) d\Pi 
\ge \frac{ \left[ 1 - \sum_{k = 1}^M \lambda_k A_{\theta^{(-k)}}(\theta^k) \right]^t - s 2^d \alpha^{-d}}{\alpha}.
\end{align}

Then by the Kantovorich-Rubinstein theorem \citep[][Theorem 1.14]{Villani2003},
\begin{align*}
\W_{\norm{\cdot}_1}(P^t(\theta, \cdot), \Pi)
&\ge \sup_{ \norm{\phi}_{\text{Lip}(\norm{\cdot}_1)} \le 1} \left[ \int \phi dP^t(\theta, \cdot) - \int \phi d\Pi \right]
\\
&\ge  \frac{ \left[ 1 - \sum_{k = 1}^M \lambda_k A_{\theta^{(-k)}}(\theta^k) \right]^t - s 2^d \alpha^{-d}}{\alpha}.
\end{align*}
Maximizing this lower bound with respect to $\alpha$, obtain
\begin{align*}
\W_{\norm{\cdot}_1}(P^t(\theta, \cdot), \Pi)
\ge \frac{\left( 1 - \frac{1}{1 + d} \right)}{2 s^{\frac{1}{d}} \left( 1 + d \right)^{\frac{1}{d}}} 
\left[ 1 - \sum_{k = 1}^M \lambda_k A_{\theta^{(-k)}}(\theta^k) \right]^{t \left(1 + \frac{1}{d} \right)}.
\end{align*}
Using Hölder's inequality \citep[Theorem 6.2]{Folland2013} completes the lower bound more generally.
\end{proof}

\end{appendix}

\bibliography{lower_bounds.bib}

\end{document}